\documentclass[reqno,11pt]{article}

\usepackage[utf8,latin1]{inputenc}
\usepackage{amsbsy}
\usepackage{amsfonts}
\usepackage{amsmath}
\usepackage{amsthm}
\usepackage{amssymb}
\usepackage{enumerate}
\usepackage{hyperref}
\usepackage{color}

\newcommand{\lcirc}{{\raise-0.15ex\hbox{$\scriptscriptstyle \circ$}}}
\newcommand{\lstar}{{\raise-0.15ex\hbox{$\scriptstyle \ast$}}}
\newcommand{\mathds}{\mathbf}
\usepackage{pdfsync}

\newtheorem{theorem}{Theorem}
\newtheorem{lemma}[theorem]{Lemma}
\newtheorem{proposition}[theorem]{Proposition}

\newtheorem{remark}[theorem]{Remark}

\newcommand{\gammab}{{\bar \gamma}}

\begin{document}
\title{The bead process for beta ensembles}
\author{Joseph Najnudel, B\'alint Vir\'ag}
\date{}
\maketitle

\begin{abstract}
The bead process introduced by Boutillier is a countable interlacing of the determinantal sine-kernel (i.e. $\operatorname{Sine}_2$) point processes.
We construct the bead process for general  $\operatorname{Sine}_{\beta}$ processes as an infinite dimensional Markov chain whose transition mechanism is explicitly described.
We show that this process is the microscopic scaling limit in the bulk of the Hermite $\beta$ corner process introduced by Gorin and Shkolnikov, generalizing 
 the process of the minors of the Gaussian Unitary and Orthogonal Ensembles.
  In order to prove our results, we use bounds on the variance of the point counting of the circular and the Gaussian beta ensembles, proven in a companion paper \cite{NV19}. 
  \end{abstract}
\section{Introduction}
In Boutillier \cite{Bou}, a remarkable family of point processes on $\mathbb Z \times \mathbb R$, called {\it bead processes}, and indexed by a parameter $\gamma \in (-1,1)$, has been defined. They enjoy the following
properties:
\begin{description}
\item[Interlacing.] The points of two consecutive lines
interlace with each other.
\item[Invariance.] The distribution of the point process is invariant and ergodic under the natural action of $\mathbb Z \times \mathbb R$ by translation.
\item[Parameters.] The expected number of points in any interval is proportional to its length. 
    Given that $(0,0)$ is in the process, the expected value of the first positive point on line $1$ is proportional to $\arccos \gamma$.
\item[Gibbs property.] The distribution of any point $X$, given the other points, is uniform on the interval
which is allowed by the interlacing property.
\end{description}
It is not known whether these properties determine 
the point process uniquely, as the closely related results of Sheffield \cite{She05} do not directly apply.

Existence was shown by Boutillier, who considers a deteminantal process with an explicit kernel. Its restriction to a line is the standard sine-kernel process. Thus the above description proposes to be the purest probabilistic definition of the Gaudin-Mehta sine kernel process limit of the bulk eigenvalues of the Gaussian Unitary Ensemble (GUE).

Boutillier's result relies on taking limits of tilings on the torus. Since then, works starting with Johansson and Nordenstam \cite{JN} showed that the consecutive minor eigenvalues of the Gaussian Unitary Ensemble also converge to the bead process, where the tilt depends on the global location within the Wigner semicircle. These results have been refined and generalized in Adler, Nordenstam and Moerbeke \cite{ANVM}. However, the
corresponding questions remained open for other
matrix ensembles, as the Gaussian Orthogonal Ensemble
(GOE), and the Gaussian Symplectic Ensemble (GSE): 
\begin{itemize}
\item Is there a limit of the eigenvalue minor process?
\item Is there a simple characterization as for $\beta=2$?
\item Can one derive formulas related to the distribution of beads?
\end{itemize}
One of the main goals of this paper is to answer positively to these questions.
The limiting process is defined as an infinite-dimensional Markov chain, the transition from one line to the next being explicitly described. This
transition can be viewed as a generalization of the limit, when the dimension
$n$ goes to infinity, of the random reflection walk on the unitary group $U(n)$. This walk is the unitary analogue of the random transposition walk studied, for
example, in Diaconis and Shahshahani \cite{DS81}, Berestycki and Durrett \cite{BD06} and Bormashenko \cite{Bor11}.

The natural generalization of the transpositions to the setting of the orthogonal group corresponds to the reflections.
The orthogonal matrix corresponding to the reflection across the plane with normal unit vector $v$ is $I-2vv^*$. To further generalize to the unitary group, we proceed as follows: given a fixed unit complex number $\eta$ and a unit vector $v$, we define
the {\bf complex reflection} across $v$ with angle $\arg(\eta)$ as the isometry whose matrix is
given by $I+(\eta -1)vv^*$. The random reflection walk $(Y_k)_{k \geq 1}$ on the unitary group $U(n)$ is then defined by $Y_k=X_1\dots X_k$, where $(X_j)_{j \geq 1}$ are independent reflections for which $v$ is chosen according to uniform measure on the complex
unit sphere, and $\eta$ is fixed.

Note that since the multiplicative increments of the walk are invariant under conjugation by any group element, it follows that $\bar Y_k$, the conjugacy class of $Y_k$, also follows a random walk. This, of course, is given by the eigenvalues of $Y_k$; the transition mechanism can be computed as follows. Assuming that the eigenvalues $u_j$ of $\bar Y_k$ are distinct, the eigenvalues of $Y_{k+1}$ are the solutions of
\begin{equation}
\sum_{j=0}^{n-1} i\frac{u_j+z}{u_j-z}\rho_j = i\frac{1+\eta}{1-\eta} \label{evolutionunitary}
\end{equation}
where for $|z|=1$ the summands and the right-hand side are both real. The only randomness is contained in the values $\rho_j$, which have a Dirichlet joint distribution with all parameters equal to 1. To summarize, in order to get the evolution of $(\bar Y_k)_{k \geq 1}$, we pick $(\rho_j)_{1 \leq j \leq n}$ from Dirichlet distribution, form the rational function given
by the left-hand side of \eqref{evolutionunitary}, and look at a particular level set to get the new eigenvalues.

This equation can be lifted to the real line. Let $(\lambda_j)_{j \in \mathbb{Z}}$ be the $(2 \pi n)$-periodic set of
$\lambda \in \mathbb{R}$ such that
$e^{i\lambda/n} \in \{u_1, \dots, u_n\}$, and extend the sequence $(\rho_j)_{1 \leq j \leq n}$ periodically (with period $n$) to all integer indices. With $z=e^{ix/n}$, the left-hand side of \eqref{evolutionunitary} can be written as
$$ \lim_{\ell \to\infty} \sum_{j=-\ell}^{\ell} \frac{2n\rho_j}{\lambda_j-x}.
$$
and the level set of this at $i(1+\eta)/(1-\eta)$ gives the lifting of the eigenvalues at the next step. Notice
now that essentially the only role of $n$ in the above process is given by the joint distribution
of the $\rho$-s. These are $n$-periodic and Dirichlet; clearly, as $n\to\infty$ they converge, after
suitable renormalization, to independent exponential variables, giving naturally an infinite-dimensional Markov chain.

In the present article, we prove rigorously the existence of this Markov chain, and we deduce a
new construction of the bead process. By replacing the exponential variables by gamma variables with general parameter, we construct a natural generalization of the
bead process, indexed by a parameter $\beta > 0$. For $\beta = 2$, this process is the bead process itself, and then it is  the scaling limit (at microscopic scale) of the eigenvalues of the GUE minors when the dimension goes to infinity. For $\beta= 1$, we show that we get the limit of the eigenvalues of the GOE minors, for $\beta = 4$, we get the limit of the eigenvalues 
of the GSE minors, 
and we generalize this result to all $\beta > 0$, by considering the Hermite $\beta$ corners, defined by Gorin and Shkolnikov \cite{GS}, which can be informally viewed as
the "eigenvalues of G$\beta$E minors". 

The sequel of the present paper is organized as follows.

In Section 2, we detail the above discussion
on the random reflection walk, and we deduce a property
of invariance for the law of the spectrum of a
Haar-distributed unitary matrix, for the
transition given by the equation \eqref{evolutionunitary}. We generalize this property
to circular beta ensembles for any $\beta > 0$.

In Section 3, we generalize the notion of Stieltjes transform
to a class of infinite point measures on the real line for which the series given by the usual definition is
not absolutely convergent.

In Section 4, we construct a family of Markov chains
on a space of point measures, for which the transition
mechanism is obtained by taking a level set of the Stieltjes
transform defined in Section 3.

In Section 5, we show how the lifting of the unit circle on the real line defined above connects the results of Section 2 to those of Section 4.

In  Section 6, we use some bound on the variance of the number of points of the circular beta ensembles in
an arc, in order to take the limit
 of the results in Section 5, when the period of the point measure goes
 to infinity. We show a property of invariance enjoyed by the determinantal sine-kernel process and
 its generalizations for all $\beta > 0$, for the Markov chain defined in Section 4. From this Markov chain,
 we deduce the construction of a stationnary point process on $\mathbb{R} \times \mathbb{Z}$, for which
 the points of a given line follow the distribution of the $\operatorname{Sine}_{\beta}$ process introduced in
 Valk\'o and Vir\'ag \cite{VV}.

 In Section 7, we show, under some technical conditions, a property of continuity of the Markov chain
 with respect to the initial point measure and the weights.

 From this result, and from a bound, proven in a companion paper \cite{NV19}
on the variance of the number of points of the Gaussian beta ensemble in intervals, we deduce in 
Section 8 that
 the generalized bead process constructed in Section 6 appears as a limit for the eigenvalues of the minors
 of Gaussian Ensembles for $\beta \in \{1,2,4\}$. The case $\beta = 2$ corresponds to the GUE, for which the
 convergence to the bead process defined by  Boutillier \cite{Bou} is already known from Adler, Nordenstam and Moerbeke \cite{ANVM}. Combining our result with \cite{ANVM} then implies that our Markov chain has necessarily the same distribution as the bead process given in \cite{Bou}.  The case
 $\beta = 1$ gives the convergence of the renormalized eigenvalues of the GOE minors, and the case $\beta = 4$ gives the convergence of the renormalized eigenvalues of the GSE minors.  For other values of $\beta$, we
 get a similar result of convergence for the renormalized points of the Hermite $\beta$ corner defined in \cite{GS}. 
\section{Random reflection chains on the unitary group}

We start with a brief review of how multiplication by complex
reflections changes eigenvalues. Let $U \in U(n)$ be a unitary matrix with distinct eigenvalues
$u_1, \dots, u_n$, and let $v$ be a unit vector. Let $a_1, \dots, a_n$ be the coefficients of $v$ in
a basis of unit eigenvectors of $U$, and let $\rho_j = |a_j|^2$ for $1 \leq j \leq n$: $\rho_1, \dots,
\rho_n$ do not depend on the choice of the eigenvector basis and the sum of these numbers is equal to $1$.

If $\eta\not= 1$ is a complex number of modulus 1, the complex reflection with angle $\arg \eta$ and vector
$v$ corresponds to the unitary matrix
$I+(\eta-1)vv^*$. If we multiply $U$ by this reflection, we get a new matrix whose eigenvalues $u$ satisfy
$$
0=\det(U(I+(\eta-1)vv^*)-u),
$$
which can be rewritten as
$$
0=\det(U-u)\det(I+(\eta-1)Uvv^*(U-u)^{^{-1}})
$$
when $u$ is not an eigenvalue of $U$. Now, the second argument is $I$ plus a rank-1 matrix, so its
determinant equals 1 plus the trace of the rank-1 matrix. Thus the equation above reduces to
$$
0=1+ (\eta-1){\rm tr}(Uvv^*(U-u)^{-1}) = 1 + (\eta-1)v^*((U-u)^{-1}U)v.
$$
Expanding $U$ in the basis of its eigenvectors and eigenvalues $u_j$, we get
$$
1=(1 - \eta) \sum_{j=1}^n \rho_j \frac{u_j}{u_j-u}
$$
or, after a transformation,
\begin{equation}
\sum_{j=1}^n i \rho_j \frac{u_j+u}{u_j-u}=i\frac{1+\eta}{1-\eta}. \label{equationinterlacing}
\end{equation}
As $u$ moves counterclockwise on the unit circle, and on each arc between two consecutive
poles, the left-hand side of \eqref{equationinterlacing} is continuous and strictly increasing from $-\infty$ to
$\infty$.
Hence, the matrix $U(I+(\eta-1)vv^*)$ has exactly one eigenvalue in each arc between
 eigenvalues of $U$: in other words, the eigenvalues of $U(I+(\eta-1)vv^*)$ strictly interlace between those of
 $U$, and are given by the solutions $u$ of the equation \eqref{equationinterlacing}.

Consider the product of the unit sphere in $\mathbb{C}^n$ and $\mathbb R$, and a distribution $\pi$ on this space
which is invariant under permutations of the $n$ coordinates of the sphere, and by multiplication of
each of these coordinates by complex numbers of modulus one.
For such a distribution, we can associate a Markov chain on unitary matrices as follows. Given
$U_0,\ldots, U_k$, we pick a sample $((a_1, \dots, a_n),h)$
from $\pi$ independently from the past. Then, $U_{k+1}$ is defined as the product of
$U_{k}$ by the reflection with parameter $\eta$ so that $h=i\frac{\eta+1}{\eta-1}$,
and vector $v=\sum a_j \varphi_j$, where $(\varphi_j)_{1 \leq j \leq n}$ are unit eigenvectors of $U$ (from
the assumption made on $\pi$, the law of
$v$ does not depend on the choice of the phases of the eigenvectors $(\varphi_j)_{1 \leq j \leq n}$).

From the discussion above, it is straightforward that if $V_k$ is the spectrum of
$U_k$, then $(V_k)_{k \geq 0}$ forms a Markov process as well; its distribution depends on
the coefficients $a_j$ only through $\rho_j$. The transition is given as follows: given $V_j$,
$(\rho_j)_{1 \leq j \leq n}$ and $h$, $V_{j+1}$ is formed by the $n$ solutions of
\eqref{equationinterlacing}.

When $a$ is uniform on the unit complex sphere of $\mathbb{C}^n$, and $h$ is independent of $a$, then
 $(\rho_j)_{1 \leq j \leq n}$ has Dirichlet$(1,\ldots,1)$ distribution, and the corresponding reflection
 is independent
of $U_k$. Thus the Markov chain reduces to a random walk:
$U_j=U_0R_1\ldots R_k$, where the reflections $(R_k)_{k \geq 1}$ are independent.

It is immediate that the Haar measure on $U(n)$ is invariant for this random walk. One deduces that if
$(\rho_j)_{1 \leq j \leq n}$ follows a Dirichlet distribution with all parameters equal to $1$, if
$h$ (and then $\eta$) is independent of $(\rho_j)_{1 \leq j \leq n}$, if the points of $V_0$ follow
the distribution of the eigenvalues of the CUE in dimension $n$, and if $(V_k)_{k \geq 0}$
is the Markov chain described above, then the law of $V_k$ does not depends of $k$: the CUE distribution
is invariant for this Markov chain.

This invariance property can be generalized to other distributions $\pi$.

Indeed, as in Simon \cite{Simon}, one can associate to
the point measure
$\sigma := \sum_{j=1}^n \rho_j \delta_{u_j}$
a so-called {\it Schur function}
 $f_{\sigma}$, which is rational, and which can be
 written, by Geronimus theorem, as
 $$f_{\sigma}(u) = R_{\alpha_0} \circ M_u \circ R_{\alpha_1} \circ M_u \circ R_{\alpha_2} \circ \cdots \circ R_{\alpha_{n-2}} \circ M_u (\alpha_{n-1}),$$
where $M_u$ denotes the multiplication by $u$, the $(\alpha_j)_{0 \leq j \leq n-1}$ are the Verblunsky coefficients associated to the orthogonal
polynomials with respect to the measure $\sigma$, and for all $\alpha \in \mathbb{D}$,
$R_{\alpha}$ is the M\"obius transformation given by
$$R_{\alpha} (z) = \frac{\alpha + z}{1 + \overline{\alpha} z}.$$
On the other hand, one has the equality of rational
functions:
 \begin{equation}
\int_{\mathbb{U}}
 i \frac{v+u}{v-u} d \sigma(v)
 =i\frac{1+u f_{\sigma}(u)}{1-u f_{\sigma} (u)}.
\end{equation}
  Hence, the equation \eqref{equationinterlacing}
  is satisfied if and only if
$u f_{\sigma}(u) = \eta$, or equivalently,
\begin{equation}
M_{\eta^{-1}}  \circ M_{u} \circ R_{\alpha_0} \circ M_u  \circ R_{\alpha_1} \circ \cdots \circ  M_u (\alpha_{n-1}) = 1. \label{MR}
\end{equation}
Now, $M_{\eta^{-1}}$ and $M_{u}$ commute and for $\alpha \in \mathbb{D}$, $M_{\eta^{-1}} \circ R_{\alpha} = R_{\alpha \eta^{-1}} \circ M_{\eta^{-1}}$.
One deduces that \eqref{MR} is equivalent to
$$M_u \circ R_{\alpha_0 \eta^{-1}} \circ M_u  \circ R_{\alpha_1 \eta^{-1}} \circ \cdots \circ  M_u (\alpha_{n-1} \eta^{-1}) = 1,$$
 i.e. $u f_{\tau} (u) = 1$, where $\tau$ is the finitely supported probability measure whose Verblunsky coefficients are
 $(\alpha_0 \eta^{-1}, \dots, \alpha_{n-1}\eta^{-1})$. Now, by the general construction of the Schur functions, the equation $u f_{\tau} (u) = 1$ is satisfied if and only if $u$ is a point of the support
 of $\tau$: in other words, this support is the set of solutions of \eqref{equationinterlacing}.
 We deduce that if the distribution $\pi$ and
 the law of $\{u_1, \dots, u_n\}$ are chosen in such
 a way that $(\alpha_0 \eta^{-1}, \dots, \alpha_{n-1}
 \eta^{-1})$ has the same law as
 $(\alpha_0, \dots, \alpha_{n-1})$, then
 the law of $\{u_1, \dots, u_n\}$ is invariant for
 the Markov chain described above. The precise statement is the following:
 \begin{proposition} \label{circularinvariance}
 Let $\pi$ be a probability
 distribution on the product of the unit sphere
 of $\mathbb{C}^n$ and $\mathbb{R}$, under which the
 first component $(a_1, \dots, a_n)$ is independent of the second $h = i (1 + \eta)/(1- \eta)$.
 We suppose that the law of
 $(a_1, \dots, a_n)$ is invariant by permutation of the coordinates, and by their pointwise multiplication by
  complex numbers of modulus $1$. Let
  $\mathbb{P}$ be a probability measure of the
  sets of $n$ points $\{u_1, \dots, u_n\}$,
  such that under the product measure $\mathbb{P}
  \otimes \pi$, the sequence $(\alpha_0, \dots,
  \alpha_{n-1})$ of Verblunsky coefficients associated
  to the measure $$\sigma = \sum_{1 \leq j \leq n}
  \rho_j \delta_{u_j} = \sum_{1 \leq j \leq n}
  |a_j|^2 \delta_{u_j}.$$
  has a law which is invariant by multiplication by
  complex numbers of modulus $1$.
  Then, the measure $\mathbb{P}$ is invariant
  for the Markov chain associated to $\pi$: more
  precisely, under $\mathbb{P}
  \otimes \pi$, the law of the set of solutions
  of \eqref{equationinterlacing} is equal to
  $\mathbb{P}$.
 \end{proposition}
 It is not obvious to find explicitly some measures $\mathbb{P}$ and $\pi$ under which the law of the Verblunsky coefficients
is invariant by rotation. An important example is obtained by considering the so-called {\it circular beta ensembles}.
These ensembles are constructed as follows: for some parameter $\beta > 0$, one defines a probability measure
$\mathbb{P}_{n,\beta}$ on the sets of $n$ points on the unit circle, such that the corresponding $n$-point correlation function $r_{n,\beta}$ is given,
for $z_1, \dots z_n \in \mathbb{U}$, by
$$r_{n, \beta} (z_1, \dots, z_N) = C_{n, \beta} \, \prod_{1 \leq j < k \leq N} |z_j - z_k|^{\beta},$$
where $C_{n, \beta} > 0$ is a normalization constant. Note that, for $\beta = 2$, one obtains the distribution of the spectrum of a random $n \times n$ unitary matrix following the
Haar measure.
Now, let $\pi_{n, \beta}$ be
any distribution on the product of the unit sphere of $\mathbb{C}^n$ and $\mathbb{R}$,
such that with the notation above, $h$ is
independent of $(\rho_0, \dots, \rho_{n-1})$, which has a Dirichlet distribution with all parameters equal to $\beta/2$.
Then, under  $\mathbb{P}_{n, \beta}
  \otimes \pi_{n, \beta}$, the distribution of
the Verblunsky coefficients $(\alpha_0, \alpha_1, \dots, \alpha_{n-1})$ has been computed in Killip and Nenciu \cite{bib:KN04}. One obtains the following:
\begin{itemize}
\item The coefficients $\alpha_0, \alpha_1, \dots \alpha_{n-1}$ are independent random variables.
\item The coefficient $\alpha_{n-1}$ is uniform on the unit circle.
\item For $j \in \{0, 1, \dots, n-2\}$, the law of $\alpha_j$ has density $(\beta/2)(n-j-1) (1-|\alpha_j|^2)^{(\beta/2)(n-j-1) - 1}$
with respect to the uniform probability measure on the unit disc: note that
$|\alpha_j|^2$ is then a beta variable of parameters
$1$ and $\beta(n-j-1)/2$.
\end{itemize}
Therefore, the law of $(\alpha_0, \alpha_1, \dots, \alpha_{n-1})$ is invariant by rotation, and one deduces the following result:
\begin{proposition} \label{CJEinvariance}
The law of the circular beta ensemble is an
invariant measure for the Markov chain associated
to $\pi_{n, \beta}$. More precisely, under
$\mathbb{P}_{n, \beta} \otimes \pi_{n, \beta}$,
the set of solutions of \eqref{equationinterlacing}
follows the distribution $\mathbb{P}_{n, \beta}$.
\end{proposition}
In the next sections, we will take a limit when $n$
goes to infinity. For this purpose, we need to
consider point processes on the real line instead
of the unit circle, and to find an equivalent of
the equation \eqref{equationinterlacing} in this
setting.
\section{Stieltjes transform for point measures}
Let $\Lambda$ be a $\sigma$-finite point measure on $\mathbb{R}$, which can be written as follows:
$$\Lambda = \sum_{\lambda \in L} \gamma_{\lambda} \delta_{\lambda},$$
where $L$ is a discrete subset of the real line, $\gamma_{\lambda} > 0$ for all $\lambda \in S$, and
$\delta_{\lambda}$ is the Dirac measure at $\lambda$.
The usual definition of the Stieltjes transform applied to $\Lambda$ gives,
for $z \in \mathbb{C} \backslash \{L\}$:
\begin{equation}
S_{\Lambda} (z) = \sum_{\lambda \in L} \frac{\gamma_{\lambda}}{\lambda - z}. \label{stieltjes}
\end{equation}
If the set $L$ is finite, then $S_{\Lambda}(z)$ is well-defined as a rational function. If $L$ is infinite and
if the right-hand side of \eqref{stieltjes} is absolutely convergent, then this equation is still meaningful.
The following result implies that under some technical assumptions, one can define $S_{\Lambda}$ even if \eqref{stieltjes} does not apply directly:
\begin{theorem} \label{theoremstieltjes}
Assume that the for all $a, b \in \mathbb{R}$, $\Lambda[0,x+a] - \Lambda[-x+b,0] = O(x/\log^2 x)$ as $x\to \infty$. Then, for
all $z \in \mathbb{C} \backslash \{L\}$, there exists
$S_{\Lambda} (z) \in \mathbb{C}$ such that
$$ \sum_{\lambda \in L \cap [-c, c]} \frac{\gamma_{\lambda}}{\lambda - z} \underset{c \rightarrow \infty}{\longrightarrow} S_{\Lambda}(z).$$
The function $S_{\Lambda}$ defined in this way is meromorphic, with simple poles at the elements of $L$, and the residue at $\lambda \in L$ is equal to
$- \gamma_{\lambda}$. The derivative of $S_{\Lambda}$ is given by
\begin{equation}
S'_{\Lambda}(z) =  \sum_{\lambda \in L} \frac{\gamma_{\lambda}}{(\lambda - z)^2}, \label{derivative}
\end{equation}
where the convergence of the series is uniform on compact sets of $\mathbb{C} \backslash \{L\}$. For all pairs $\{\lambda_1, \lambda_2\}$ of consecutive
points in $L$, with $\lambda_1 < \lambda_2$, the function $S_{\Lambda}$ is a strictly
increasing bijection from $(\lambda_1, \lambda_2)$ to $\mathbb{R}$. Moreover, we have the following translation invariance: if $y\in \mathbb R$ and $\Lambda$ satisfies the conditions above, then so does its translation  $\Lambda+y$, and one has
$$S_{\Lambda + y}(z + y) = S_{\Lambda} (z)$$
for all $z \in \mathbb{C} \backslash \{L\}$.
\end{theorem}
\begin{remark}
The bound $x/\log^2 x$ is somohow arbitrary and not optimal (any increasing function which is negligible with respect to $x$ and integrable  against $dx/x^2$ at infinity would work). However, it will be sufficient for our purpose.
\end{remark}
\begin{proof}
Let $c_0 > 1$, and $z \in \mathbb{C}$ such that $|z| \leq c_0/2$. For $c > c_0$, we have:
\begin{align*}
\sum_{\lambda \in L \cap ([-c,-c_0] \cup [c_0,c] )} \frac{\gamma_{\lambda}}{\lambda - z }
& = \sum_{\lambda \in L \cap [c_0,c]} \gamma_{\lambda} \int_{\lambda}^{\infty} \frac{d \mu}{(\mu -z)^2}
-  \sum_{\lambda \in L \cap [-c,-c_0]} \gamma_{\lambda} \int_{-\infty}^{\lambda} \frac{d \mu}{(\mu -z)^2}
\\ & = \int_{c_0}^{\infty} \frac{ \Lambda([c_0, c \wedge \mu])}{(\mu -z)^2} \, d\mu
-  \int_{-\infty}^{-c_0} \frac{ \Lambda([(-c) \vee \mu, - c_0])}{(\mu -z)^2} \, d\mu
\\ & =  \int_{c_0}^{\infty} \left( \frac{ \Lambda([c_0, c \wedge \mu])}{(\mu -z)^2} -  \frac{ \Lambda([-(c \wedge \mu), - c_0])}{(\mu +z)^2} \right)\, d \mu
\\ & = \int_{c_0}^{\infty} \frac{ \Lambda([c_0, c \wedge \mu]) - \Lambda([-(c \wedge \mu), - c_0])}{\mu^2} \, d \mu
\end{align*}
$$ +  \int_{c_0}^{\infty} \left( \frac{ (2 z \mu - z^2) (\Lambda([c_0, c \wedge \mu]) )}{\mu^2 (\mu-z)^2}
+ \frac{(2 z \mu + z^2) (\Lambda([-(c \wedge \mu), -c_0]) )}{\mu^2 (\mu+z)^2}
 \right) \, d\mu.$$
 Let $F$ be an increasing function from $\mathbb{R}_+$ to $\mathbb{R}_+^*$, such that $F(x)$ is equivalent to $x/\log^2 x$ when $x$ goes to infinity.
  By assumption, there exists $C > 0$ such that for all $x \geq 0$,
$|\Lambda([0, x]) - \Lambda([-x, 0])| \leq C F(x) $, and
then, for all $\mu \geq c_0$,
$$ |\Lambda([c_0, c \wedge \mu]) - \Lambda([-(c \wedge \mu), - c_0])| \leq C F(c \wedge \mu) + \Lambda([-c_0, c_0])
\leq \left( C + \frac{ \Lambda([-c_0, c_0]) }{F(0)} \right) \, F(\mu).$$  Since $\mu \mapsto F(\mu)/\mu^2$ is integrable at infinity,
one obtains, by dominated convergence,
$$ \int_{c_0}^{\infty} \frac{ \Lambda([c_0, c \wedge \mu]) - \Lambda([-(c \wedge \mu), - c_0])}{\mu^2} \, d \mu
\underset{c \rightarrow \infty}{\longrightarrow}  \int_{c_0}^{\infty} \frac{ \Lambda([c_0,  \mu]) - \Lambda([-\mu, - c_0])}{\mu^2} \, d \mu,$$
where the limiting integral is absolutely convergent. Similarly, there exist $C', C'' > 0$ such that for all
$x \geq 0$, $|\Lambda([0,x+1]) - \Lambda([-x,0])| \leq C' F(x)$ and $|\Lambda([0,x]) - \Lambda([-x-1,0])| \leq C'' F(x)$,  which implies that
\begin{align*}
\Lambda((x,x+1]) + \Lambda([-x-1,-x))  & \leq |\Lambda([0,x+1]) - \Lambda([-x,0])| + |\Lambda([0,x]) - \Lambda([-x-1,0])|  \\ &  \leq (C'+C'') F(x).
\end{align*}
Hence, for all integers $n \geq 1$,
\begin{align*}
\Lambda([-n,n]) & = \Lambda(\{0\}) + \sum_{k=0}^{n-1} (\Lambda((k,k+1]) + \Lambda([-k-1,-k)) \\ & \leq  \Lambda(\{0\}) + (C' + C'') \, \sum_{k=0}^{n-1} F(k)
\leq K n F(n-1)
\end{align*}
where $K > 0$ is a constant, and then for all $x \geq 0$, $\Lambda([-x,x]) \leq K (1+x) F(x)$, which implies that for $\mu \geq c_0$,
$ \Lambda([-(c \wedge \mu), -c_0]) \leq K (1+ \mu) F(\mu)$ and $\Lambda([c_0, c \wedge \mu]) \leq K(1+ \mu) F(\mu)$.
Moreover, since $|z| \leq c_0/2 \leq \mu/2$, one has $|\mu - z| \geq \mu/2$,  $|\mu + z| \geq \mu/2$  and
\begin{equation}
\left| \frac{ (2 z \mu - z^2)}{\mu^2 (\mu-z)^2} \right| + \left| \frac{ (2 z \mu +  z^2)}{\mu^2 (\mu+z)^2} \right|
\leq 2 \, \frac{  2.5 |z| \mu}{ \mu^2 (\mu/2)^2} = 20 |z| / \mu^3 \leq 10 \, c_0/\mu^3 \label{20z}
\end{equation}
Since $\mu \mapsto (1+ \mu) F(\mu) / \mu^3$ is integrable at infinity, one can again apply dominated convergence and obtain that
$$\int_{c_0}^{\infty} \left( \frac{ (2 z \mu - z^2) (\Lambda([c_0, c \wedge \mu]) )}{\mu^2 (\mu-z)^2}
+ \frac{(2 z \mu + z^2) (\Lambda([-(c \wedge \mu), -c_0]) )}{\mu^2 (\mu+z)^2}
 \right) \, d\mu$$
 tends to $$\int_{c_0}^{\infty} \left( \frac{ (2 z \mu - z^2) (\Lambda([c_0, \mu]) )}{\mu^2 (\mu-z)^2}
+ \frac{(2 z \mu + z^2) (\Lambda([-\mu, -c_0]) )}{\mu^2 (\mu+z)^2}
 \right) \, d\mu$$ when $c$ goes to infinity. Therefore,
 \begin{align*}
 \sum_{\lambda \in L \cap ([-c,-c_0] \cup [c_0,c] )} \frac{\gamma_{\lambda}}{\lambda - z }
& \underset{c \rightarrow \infty}{\longrightarrow}
  \int_{c_0}^{\infty} \frac{ \Lambda([c_0,  \mu]) - \Lambda([-\mu, - c_0])}{\mu^2} \, d \mu
  \\ & + \int_{c_0}^{\infty} \left( \frac{ (2 z \mu - z^2) (\Lambda([c_0, \mu]) )}{\mu^2 (\mu-z)^2}
+ \frac{(2 z \mu + z^2) (\Lambda([-\mu, -c_0]) )}{\mu^2 (\mu+z)^2}
 \right) \, d\mu,
\end{align*}
 which proves the existence of the limit defining $S_{\Lambda} (z)$: explicitly, for $z \in \mathbb{C} \backslash \{L\}$ and for any $c_0 > 2|z| \vee 1$,
 \begin{align}
 S_{\Lambda}(z) & = \sum_{\lambda \in L \cap (-c_0,c_0)}  \frac{\gamma_{\lambda}}{\lambda - z }
+ \int_{c_0}^{\infty} \frac{ \Lambda([c_0,  \mu]) - \Lambda([-\mu, - c_0])}{\mu^2} \, d \mu \nonumber
\\ & +  \int_{c_0}^{\infty} \left( \frac{ (2 z \mu - z^2) (\Lambda([c_0, \mu]) )}{\mu^2 (\mu-z)^2}
+ \frac{(2 z \mu + z^2) (\Lambda([-\mu, -c_0]) )}{\mu^2 (\mu+z)^2}
 \right) \, d\mu. \label{int}
 \end{align}
 For fixed $c_0 > 0$, the first term of \eqref{int} is a rational function of $z$, the second term of \eqref{int} does not depend on $z$, and by dominated convergence,
 the third term can be differentiated in the integral if we restrict $z$ to the
 set $\{|z| < c_0/2\}$. Hence, the restriction of $S_{\Lambda}$ to the set $\{z  < c_0/2\}$ is meromorphic, with simple poles at points $\lambda \in L \cap (-c_0/2, c_0/2)$.
 Since $c_0$ can be taken arbitrarily large, $S_{\Lambda}$ is in fact meromorphic on
 $\mathbb{C}$, with poles $\lambda \in L$, the pole $\lambda$ having residue $- \gamma_{\lambda}$.
 The derivative $S'_{\Lambda}(z)$ is given, for any $c_0 > 2 |z| \vee 1$, by:
 \begin{align}
 S'_{\Lambda} (z) &  =  \sum_{\lambda \in L \cap (-c_0,c_0)}  \frac{\gamma_{\lambda}}{(\lambda - z)^2 }
+ 2 \,  \int_{c_0}^{\infty} \left( \frac{ (\Lambda([c_0, \mu]) )}{ (\mu-z)^3}
+ \frac{ (\Lambda([-\mu, -c_0]) )}{ (\mu+z)^3}
 \right) \, d\mu. \nonumber \\ &
 =  \sum_{\lambda \in L \cap (-c_0,c_0)}  \frac{\gamma_{\lambda}}{(\lambda - z)^2 }  +
   \int_{c_0}^{\infty}   \, \left( \sum_{\lambda \in L \cap [c_0, \mu]} \gamma_{\lambda} \right) \,  \frac{2 \, d \mu}{(\mu-z)^3}
 +   \int_{c_0}^{\infty}   \, \left( \sum_{\lambda \in L \cap [-\mu, -c_0]} \gamma_{\lambda} \right) \,  \frac{2 \, d \mu}{(\mu+z)^3}
 \nonumber \\
 &  =  \sum_{\lambda \in L \cap (-c_0,c_0)}  \frac{\gamma_{\lambda}}{(\lambda - z)^2 }   +
 \sum_{\lambda \in L \cap [c_0, \infty)}  \gamma_{\lambda}  \int_{\lambda}^{\infty} \, \frac{2 \, d \mu}{(\mu-z)^3}
 + \sum_{\lambda \in L \cap (-\infty, c_0]} \gamma_{\lambda} \int_{-\lambda}^{\infty}  \, \frac{2 \, d \mu}{(\mu+z)^3},
 \nonumber
 \end{align}
 which implies \eqref{derivative}. Note that the implicit use of Fubini theorem is this computation is correct since all the sums and integral involved are
 absolutely convergent.

 Now, let $\mathcal{K}$ be a compact set of $\mathbb{C} \backslash L$, let $d > 0$ be the distance between $\mathcal{K}$ and $L$, and let
 $A > 0$ be the maximal modulus of the elements of $\mathcal{K}$. For all $z \in \mathcal{K}$ and $\lambda \in L$, one has, for $|\lambda| \leq 2A+1$,
$$ \left| \frac{\gamma_{\lambda} }{(\lambda - z)^2} \right| \leq \frac{\gamma_{\lambda}}{d^2} \leq \frac{1+ (2A+1)^2}{d^2} \cdot \frac{\gamma_{\lambda}}{1 + \lambda^2} $$
and for $|\lambda| \geq 2A+1$,
$$\left|  \frac{\gamma_{\lambda} }{(\lambda - z)^2} \right| \leq \frac{\gamma_{\lambda} }{ (|\lambda| - A)^2}
\leq \frac{4 \gamma_{\lambda} }{ \lambda^2} \leq  \frac{8 \gamma_{\lambda} }{1 + \lambda^2}.$$
 Hence, in order to prove the uniform convergence of \eqref{derivative} on compact sets, it is sufficient
 to check that $$\sum_{\lambda \in L} \frac{\gamma_{\lambda} }{1 + \lambda^2} < \infty,$$
 but this convergence is directly implied by the absolute convergence of the right-hand side of \eqref{derivative} for any particular value
 $z \in \mathbb{C} \backslash L$.

 The formula \eqref{derivative} applied to $z \in \mathbb{R}$ implies immediately that for all pairs $\{\lambda_1, \lambda_2\}$ of consecutive
points in $L$, with $\lambda_1 < \lambda_2$, the function $S_{\Lambda}$ is strictly
increasing on the interval $(\lambda_1, \lambda_2)$. Moreover, one has for $\lambda \in \{\lambda_1, \lambda_2\}$ and $z \rightarrow \lambda$,
$S_{\lambda}(z) \sim \gamma_{\lambda}/(\lambda - z)$, which implies that $S_{\Lambda}(z) \rightarrow -\infty$ for $z \rightarrow \lambda_1$ and
$z >\lambda_1$, and  $S_{\Lambda}(z) \rightarrow +\infty$ for $z \rightarrow \lambda_2$ and
$z <\lambda_2$. We deduce that $S_{\Lambda}$ is a bijection from $(\lambda_1, \lambda_2)$ to $\mathbb{R}$.

 It only remains to show the invariance by translation. If we fix $y \in \mathbb{R}$, then for all $a, b \in \mathbb{R}$, and for $x \geq 0$ large enough,
\begin{align*}
(\Lambda+y)([0,x+a]) & - (\Lambda+ y)([-x+b,0])  = \Lambda([-y, x+a-y])-  \Lambda([-x + b-y, -y]) \\ &
 =  \Lambda([0, x+a-y])-  \Lambda([-x + b-y, 0]) + O(\Lambda([-|y|,|y|])) \\ & =
 O(x/\log^2 x) + O(1) = O(x/\log^2 x),
 \end{align*}
 and the assumptions of Theorem \ref{theoremstieltjes} are satisfied. One has
 $$\Lambda + y = \sum_{\lambda \in L} \gamma_{\lambda} \delta_{\lambda + y},$$
 and then for all $z \in \mathbb{C} \backslash L$,
$$ S_{\Lambda + y} (z+y)  =  \lim_{c \rightarrow \infty} \, \sum_{\lambda \in (L+y) \cap [-c,c]} \, \frac{ \gamma_{\lambda-y} }{z + y- \lambda }
  =  \lim_{c \rightarrow \infty} \, \sum_{\lambda \in L \cap [-c-y,c-y]} \, \frac{ \gamma_{\lambda} }{z -  \lambda }, $$
which is equal to $S_{\Lambda}(z)$, provided that we check that
 $$\sum_{\lambda \in L \cap [-c-y,c-y]} \, \frac{ \gamma_{\lambda} }{z -  \lambda } - \sum_{\lambda \in L \cap [-c,c]} \, \frac{ \gamma_{\lambda} }{z -  \lambda }
 \underset{c \rightarrow \infty}{\longrightarrow} 0,$$
which is implied by
\begin{equation}
\sum_{\lambda \in L \cap  [-c-|y|,-c + |y|] }  \frac{ \gamma_{\lambda} }{|z -  \lambda| } + \sum_{\lambda \in L \cap  [c-|y|,c + |y|] }  \frac{ \gamma_{\lambda} }{|z -  \lambda| }
\underset{c \rightarrow \infty}{\longrightarrow} 0.  \label{conv}
\end{equation}
Now, for $c > |y|+|z| + 1$, the left-hand side of \eqref{conv} is smaller than or equal to
\begin{align*}
& \frac{\Lambda([-c-|y|, -c + |y|]) + \Lambda([c-|y|, c+ |y|])}{c -|z| - |y|} \\ \leq &
 \frac{|\Lambda([0,c+|y|]) - \Lambda([-c+|y|+1, 0])| + |\Lambda([0,c-|y|-1]) - \Lambda([-c-|y|,0])|}{ c-|y|-|z|} = O(1/\log^2 c),
 \end{align*}
for $c$ tending to infinity.
\end{proof}
The assumption of Theorem \ref{theoremstieltjes} depends on the fact that the measure $\Lambda$ is not too far from being symmetric with respect to a given point on
the real line.
The next proposition expresses this assumption in terms of the support $L$ of $\Lambda$ and the weights $(\gamma_{\lambda})_{\lambda \in L}$.
The following result gives a sufficient condition for Theorem \ref{theoremstieltjes}:
\begin{proposition} \label{propositionstieltjes}
Consider the measure
$$
\Lambda= \sum_{j\in \mathbb Z} {\gamma_j} \delta_{\lambda_j}
$$
where $(\lambda_j)_{j \in \mathbb{Z}}$ is strictly increasing and neither bounded from above nor from below, and $\gamma_j\in \mathbb{R}_+^*$. Let $L$ be the set $\{\lambda_j, j \in \mathbb{Z}\}$. Assume that for some $c > 0$,
$$\sum_{j=0}^k \gamma_j =  ck + O(k/\log^2 k) \; \; \mathrm{   and   } \; \; \sum_{j=0}^k \gamma_{-j} =  ck + O(k/\log^2 k),$$
when $k\to\infty$. If for $x\to\infty$ one has $ \operatorname{Card} (L \cap [0,x] ) = O(x)$ and for all $a, b \in \mathbb{R}$,
$\operatorname{Card} (L \cap [0,x+a] ) - \operatorname{Card} (L \cap [-x+b,0] ) = O(x/\log^2 x)$, then the assumptions of Theorem \ref{theoremstieltjes} are satisfied.
\end{proposition}
\begin{proof}
For $y \in \mathbb{R}$, let $N(y)$ (resp. $N(y-)$) be the largest index $j$ such that $\lambda_j \leq y$ (resp. $\lambda_j < y$).
One has, for $a, b \in \mathbb{R}$ and for $x$ large enough, $$\Lambda([0,x+a]) = \sum_{j=N(0-) +1}^{N(x+a)} \gamma_j \; \; \mathrm{   and   } \; \;
\Lambda([-x+b,0]) =  \sum_{j=N((-x+b)-)}^{N(0)} \gamma_j,$$
which implies that for $x \rightarrow \infty$ and then $ N(x+a) \rightarrow \infty$, $N((-x+b)-) \rightarrow - \infty$:
\begin{align*}
\Lambda([0,x+a])  - \Lambda([-x+b,0]) & = c ( N(x+a) - |N((-x+b)-)|) \\ & + O\left(\frac{N(x+a)}{ \log^2 (N(x+a))} + \frac{|N((-x+b)-)|}{ \log^2 |N((-x+b)-)|} \right).
\end{align*}
Now, we have the following estimates: $$N(x+a) = \operatorname{Card} (L \cap [0,x+a] )  + O(1) = O(x+a) + O(1) = O(x),$$
  $$\frac{N(x+a)}{ \log^2 (N(x+a))}  = O(x/\log^2 x),$$
$$N(x+a) - |N((-x+b)-)| = \operatorname{Card} (L \cap [0,x+a] ) - \operatorname{Card} (L \cap [-x+b,0] ) + O(1) = O(x/\log^2 x),$$
$$ |N((-x+b)-)| \leq N(x+a) + |N(x+a) - |N((-x+b)-)| | \leq O(x) + O(x/\log^2 x) = O(x)$$
and
$$ \frac{|N((-x+b)-)|}{ \log^2 |N((-x+b)-)|}= O(x/\log^2 x).$$
Putting all together gives:
$$\Lambda([0,x+a])  - \Lambda([-x+b,0])  = O(x/\log^2 x)$$
and then the assumptions of Theorem \ref{theoremstieltjes} are satisfied.
\end{proof}
As written in the statement of Theorem \ref{theoremstieltjes}, the function $S_{\Lambda}$ induces a bijection between each interval $(\lambda_1,\lambda_2)$, $\lambda_1$ and
$\lambda_2$ being two consecutive points of $L$, and the real line. It is then natural to study the inverse of this bijection, which should map each element of $\mathbb{R}$ to
a set of points interlacing with $L$. The precise statement we obtain is the following:

\begin{proposition} \label{L'}
Let $\Lambda$ be a measure, whose support $L$ is neither bounded from above nor from below, and satisfying the assumptions of Theorem \ref{theoremstieltjes}.
Then, for all $h \in \mathbb{R}$, the set $S^{-1}_{\Lambda} (h)$ of $z \in \mathbb{C} \backslash\{L\}$ such that $S_{\Lambda}(z) = h$ is included in $\mathbb{R}$, and
interlaces with $L$, i.e. it contains exactly one point in each open interval between two consecutive points of $L$. Moreover, if $\Lambda$ satisfies the assumptions of Proposition \ref{propositionstieltjes}, then
it is also the case for the set $L' := S^{-1}_{\Lambda} (h)$, i.e. for $x$ going to infinity, one has $\operatorname{Card} (L' \cap [0,x] )  = O(x)$ and for all $a, b \in \mathbb{R}$,
$\operatorname{Card} (L' \cap [0,x+a] ) - \operatorname{Card} (L' \cap [-x+b,0] ) = O(x/\log^2 x)$.
\end{proposition}
\begin{proof}
The interlacing property of points of $S^{-1}_{\Lambda}(h) \cap \mathbb{R}$ comes from the discussion above, so the first part of the proposition is proven if we check
that $S_{\Lambda}(z) \notin \mathbb{R}$ if $z \notin \mathbb{R}$. Now, for all $z \in \mathbb{C} \backslash  L$,
\begin{align*}
\Im \left( S_{\Lambda}(z) \right) =
\lim_{c \rightarrow \infty} \sum_{\lambda \in L \cap [-c,c]} \Im \left( \frac{\gamma_{\lambda}}{\lambda -z} \right)
& = \lim_{c \rightarrow \infty} \sum_{\lambda \in L \cap [-c,c]} \frac{ - \gamma_{\lambda} \, \Im (\lambda-z)}{ \Re^2(\lambda-z) + \Im^2(\lambda- z) }\\
& = \lim_{c \rightarrow \infty} \sum_{\lambda \in L \cap [-c,c]} \frac{ \gamma_{\lambda} \Im (z)}{ \Re^2(\lambda-z) + \Im^2(z) }.
\end{align*}
If $z \notin \mathbb{R}$, each term of the last sum is nonzero and has the same sign as $\Im(z)$. One deduces that
 $\Im \left( S_{\Lambda}(z) \right)$ has the same properties, and then $S_{\Lambda}(z) \notin \mathbb{R}$.

 Now, the interlacing property implies that for any finite interval $I$, $$| \operatorname{Card} (L' \cap I) -  \operatorname{Card} (L \cap I) | \leq 1.$$
 If $\Lambda$ satisfies the assumptions of Proposition \ref{propositionstieltjes}, then for $a, b \in \mathbb{R}$ and for $x$ going to infinity,
 $$ \operatorname{Card} (L' \cap [0,x] )  =  \operatorname{Card} (L \cap [0,x] ) + O(1) = O(x) + O(1) = O(x)$$
 and
 \begin{align*}
 \operatorname{Card} (L' \cap [0,x+a] ) & - \operatorname{Card} (L' \cap [-x+b,0] )   = \operatorname{Card} (L \cap [0,x+a] ) \\ &  - \operatorname{Card} (L \cap [-x+b,0] ) + O(1)
  = O(x/\log^2x).
 \end{align*}
 \end{proof}

 Proposition \ref{L'} shows that the Stieltjes transform gives a way to construct a discrete subset of $\mathbb{R}$ from another, provided that we get
 a family $(\gamma_j)_{j \in \mathbb{Z}}$ of weights and a parameter $h \in \mathbb{R}$. In the next section, we use and randomize this procedure in order to
define a family of Markov chains satisfying some remarkable properties.
 \section{Stieltjes Markov chains}
 In order to put some randomness in the construction above, we need to define precisely a measurable space in which the point processes will be contained.
 The choice considered here is the following:
 \begin{itemize}
 \item We define $\mathcal{L}$ as the family of all the discrete subsets $L$ of $\mathbb{R}$, unbounded from above and from below, and satisfying the assumptions of
 Proposition \ref{propositionstieltjes}, i.e. for $x$
 going to infinity, $ \operatorname{Card} (L \cap [0,x] ) = O(x)$ and for all $a, b \in \mathbb{R}$,
$\operatorname{Card} (L \cap [0,x+a] ) - \operatorname{Card} (L \cap [-x+b,0] ) = O(x/\log^2 x)$.
\item We define, on $\mathcal{L}$, the $\sigma$-algebra $\mathcal{A}$ generated by the maps $L \mapsto \operatorname{Card} (L  \cap I)$ for all open, bounded intervals $I \subset
\mathbb{R}$, which is also the $\sigma$-algebra generated by the maps $L \mapsto  \operatorname{Card} (L  \cap B)$ for all Borel sets $B \subset \mathbb{R}$.
 \end{itemize}
 A similar choice of measurable space has to be made for the weights $(\gamma_j)_{j \in \mathbb{Z}}$:
 \begin{itemize}
 \item We define $\Gamma$ as the family of doubly infinite sequences $(\gamma_j)_{j \in \mathbb{Z}}$ satisfying the assumptions of Proposition \ref{propositionstieltjes}, i.e.
 for $k$ going to infinity,
 $$\sum_{j=0}^k \gamma_j =  ck + O(k/\log^2 k) \; \; \mathrm{   and   } \; \; \sum_{j=0}^k \gamma_{-j} =  ck + O(k/\log^2 k),$$
where $c > 0$ is a constant.
\item We define, on $\Gamma$, the $\sigma$-algebra $\mathcal{C}$ generated by the coordinate maps $\gamma_j$, $j \in \mathbb{Z}$.
\end{itemize}

Let $\mathcal{D}$ be the map from $\mathcal{L} \times \Gamma \times \mathbb{R}$ to $\mathcal{L}$, defined by:
$$\mathcal{D}(L, (\gamma_j)_{j \in \mathbb{Z}}, h) = S^{-1}_{ \sum_{j \in \mathbb{Z}} \gamma_j \delta_{\lambda_j}} (h),$$
where $\lambda_j$ is the unique increasing labeling of $L$ so that $\lambda_{-1}<0\le \lambda_0$.
Proposition \ref{L'} shows that this is indeed a map to $\mathcal L$. It is easy to show that $\mathcal D$ is measurable. Now for any probability measure $\Pi$ on $\Gamma\times \mathbb R$, it naturally defines a Markov chain $(X_k)_{k \geq 0}$ on $\mathcal L$. To get $X_{k+1}$ from $X_k$, just take a fresh sample $G_k$ (independent of $X_k$ and its past) and set $$X_{k+1}=\mathcal D(X_k,G_k).$$
By construction, $X_k$ is then a time-homogeneous Markov chain.

Clearly, if the distribution of $X_0$ is invariant under translations of $\mathbb R$, and the distribution of the $((\gamma_j)_{j \in \mathbb{Z}},h)$ in $G_k$ is invariant under translations of the indices $j$, it follows that $X_1$ also has a translation-invariant distribution.

There are two important examples of probability measures $\Pi$ for which this construction applies:
\begin{itemize}
\item Under $\Pi$, $(\gamma_j)_{j \in \mathbb{Z}}$ is a family of i.i.d, square-integrable random variables, and $h$ is independent of $(\gamma_j)_{j \in \mathbb{Z}}$.
\item Under $\Pi$, $(\gamma_j)_{j \in \mathbb{Z}}$ is a family of random variables, $n$-periodic for some $n \geq 1$, such that
$(\gamma_0, \gamma_1, \dots, \gamma_{n-1}) = (\gamma_1, \dots, \gamma_{n-1}, \gamma_0)$ in law, and $h$ is independent of $(\gamma_j)_{j \in \mathbb{Z}}$.
\end{itemize}
The fact that  $(\gamma_j)_{j \in \mathbb{Z}}$ is almost surely in $\Gamma$ comes from the law of the iterated logarithm in the first example, and directly from the
periodicity in the second example.

\section{Periodic Stieltjes Markov chains}

Consider the case when
$$\Lambda =\sum_{j\in \mathbb Z} \gamma_j \delta_{\lambda_j}$$
is invariant by translation of  $2\pi n$, and when there are $n$ point masses in every interval of length $2\pi n$ with total weight $2n$. In this case, $\Lambda$ can be thought as $2n$ times the lifting of the measure
$$
\sigma=\sum_{j=0}^{n-1} \frac{\gamma_j}{2n} \delta_{e^{i{\lambda_j}/n}}
$$
on the unit circle $\mathbb U$ under a covering map. Moreover, with $u=e^{iz/n}$ the Stieltjes transform of $\Lambda$ can be expressed in terms of $\sigma$ by
$$
S_\Lambda(z) = \sum_{j=0}^{n-1} i \frac{\gamma_j}{2n} \frac{e^{i\lambda_j/n}  + u}{e^{i\lambda_j/n} - u}.
$$
Indeed, periodicity implies that for $z\notin L$, we have
\begin{align*}
S_\Lambda (z) &  =  \underset{k \rightarrow \infty}{\lim} \sum_{j = - k n}^{ kn-1}  \frac{\gamma_j}{\lambda_j - z}
= \underset{k \rightarrow \infty}{\lim} \sum_{j=0}^{n-1}  \gamma_j \, \left( \sum_{\ell=-k}^{k-1} \frac{1}{2 \pi n\ell+ \lambda_j - z} \right)
\\ & = \sum_{j=0}^{n-1} \gamma_j  \left( \underset{k \rightarrow \infty}{\lim} \sum_{\ell=-k}^{k-1} \frac{1}{2 \pi n \ell + \lambda_j - z} \right)
= \frac{1}{2n} \,  \sum_{j=0}^{n-1} \gamma_j \cot \left(\frac{\lambda_j - z}{2n} \right).
\end{align*}
Therefore, if we set $\rho_j := \gamma_j/2n$ and
$u_j = e^{i \lambda_j/n}$, we can check that
$\mathcal{D} (L, (\gamma_n)_{n \in \mathbb{Z}}, h)$
is the set of $z \in \mathbb{R}$, such that
$e^{iz/n}$ satisfies \eqref{equationinterlacing},
for $h = i(1+\eta)/(1-\eta)$.

This property shows that the lifting $u \mapsto
\{z \in \mathbb{R}, e^{iz/n} = u  \}$ from $\mathbb{U}$
to $\mathbb{R}$ defined above transforms the Markov
chain defined in Section 2 to the Markov chain defined
in Section 4. In particular, from
Propositions \ref{circularinvariance} and \ref{CJEinvariance}, we deduce the following results:
\begin{theorem}
Let $\Pi$ be a probability measure on the space $(\Gamma \times \mathbb{R}, \mathcal{C} \otimes \mathcal{B}(\mathbb{R}))$, under which the following
holds, for some integer $n \geq 1$:
\begin{itemize}
\item Almost surely under $\Pi$, $(\gamma_n)_{n \in \mathbb{Z}}$ is $n$-periodic, and $\sum_{j=0}^{n-1} \gamma_j = 2n$.
\item The law of $(\gamma_0, \dots, \gamma_{n-1})$ is invariant by permutation of the coordinates.
\item The sequence $(\gamma_j)_{j \in \mathbb{Z}}$ is independent of $h$.
\end{itemize}
Let $\mathbb{Q}$ be a probability on $(\mathcal{L}, \mathcal{A})$ under which almost surely, the set $L$ is $(2n \pi)$-periodic and
contains exactly $n$ points in the interval $[0, 2 \pi n)$: in this case, there exists a sequence $( u_1, \dots, u_{n})$ of elements of $\mathbb{U}$,
with increasing argument in $[0, 2\pi)$, and such that
$$L = \{ z \in \mathbb{R}, e^{iz/n} \in \{ u_1, \dots, u_{n}\} \}.$$
Under the probability $\mathbb{Q} \otimes \pi$, one can define a random probability measure $\sigma$ on the unit circle by:
$$\sigma:= \frac{1}{2n} \sum_{j=1}^{n} \gamma_j \delta_{u_j}.$$
Let us assume that the joint law of the Verblunsky coefficients $(\alpha_0, \dots, \alpha_{n-1})$ of $\sigma$ is invariant by rotation, i.e.
for all $u \in \mathbb{U}$,
$$(\alpha_0 u, \dots \alpha_{n-1} u) = (\alpha_0, \dots, \alpha_{n-1})$$
in distribution.
Then, the probability measure $\mathbb{Q}$ is an invariant measure for the Markov chain associated to $\pi$.
\end{theorem}

\begin{theorem} \label{betafinite}
Let $\beta > 0$, $n \geq 1$, and let $\Pi$ be a probability measure under which the following holds almost surely:
\begin{itemize}
\item The sequence $(\gamma_n)_{n \in \mathbb{Z}}$ is $n$-periodic.
\item The tuple $(\gamma_0/2n, \dots, \gamma_{n-1}/2n)$ follows a Dirichlet distribution with all parameters equal to $\beta/2$.
\item The sequence $(\gamma_j)_{j \in \mathbb{Z}}$ is independent of $h$.
\end{itemize}
Let $\mathbb{Q}_{n, \beta}$ be the distribution of the set
$$\{z \in \mathbb{R} , e^{iz/n} \in V\},$$
where $V$ is a subset of $\mathbb{U}$ following $\mathbb{P}_{n,\beta}$, i.e. a circular beta ensemble with parameter $\beta$.
Then, $\mathbb{Q}_{n, \beta}$ is an invariant measure for the Markov chain associated to $\Pi$.
\end{theorem}
In the next section, we will let $n \rightarrow \infty$ and we will obtain a similar result in which the variables $(\gamma_n)_{n \geq 1}$ will be independent and identically distributed.

\section{An invariant measure for independent gamma random variables}

In Theorem \ref{betafinite}, we have found an invariant measure on $\mathcal{L}$, corresponding to a measure $\Pi$ under which the sequence $(\gamma_j)_{j \in \mathbb{Z}}$
is periodic, each period forming a renormalized Dirichlet distribution. For $n \geq 1$ and $\beta > 0$ fixed, and under $\Pi$, the sequence $(\gamma_j)_{j \in \mathbb{Z}}$ can be written
in function of a sequence $(g_j)_{j \in \mathbb{Z}}$ of i.i.d Gamma variables with parameter $\beta/2$, as follows:
$$\gamma_{j} = \frac{2 n g_k}{\underset{-n/2 < \ell \leq n/2}{\sum} g_{\ell}},$$
where $-n/2 < k \leq n/2$ and $k  \equiv j$ modulo $n$. For $\beta$ fixed, if we construct the sequence $(\gamma_j)_{j \in \mathbb{Z}}$ for all values of $n$, starting with the same
sequence $(g_j)_{j \in \mathbb{Z}}$, we obtain, by the law of large numbers, that for all $j \in \mathbb{Z}$, $\gamma_j$ tends almost surely to $4 g_j/\beta$ when $n$ goes to infinity.
Hence, if we want to make $n \rightarrow \infty$ in Theorem \ref{betafinite}, we should consider a measure $\Pi$ under which $(\beta \gamma_j/4)_{j \in \mathbb{Z}}$ is
a sequence of i.i.d. Gamma random variables of parameter $\beta/2$.

On the other hand, for $n$ going to infinity, the probability $\mathbb{Q}_{n, \beta}$ converges to a limiting measure $\mathbb{Q}_{\beta}$, which is the distribution of the so-called $\mathrm{Sine}_{\beta}$ {\it point process}, constructed in \cite{bib:KSt09} and
\cite{VV}. 

Therefore, taking the limit $n \rightarrow \infty$ in Theorem \ref{betafinite} suggests the following result, whose proof is given below:
\begin{theorem} \label{betainfinite}
Let $\beta > 0$, and let $\Pi$ be a probability measure under which the random variables $h$ and $(\gamma_j)_{j
 \in \mathbb{Z}}$ are all independent, $\gamma_j$ being
 equal to $4/\beta$ times a gamma random variable of parameter $\beta/2$. Then, the law $\mathbb{Q}_{\beta}$ of the $\mathrm{Sine}_{\beta}$ point process is
carried by the space $\mathcal{L}$ and it is an invariant measure for the Markov chain associated to $\Pi$.
\end{theorem}

\begin{remark}
Since the variables $(\gamma_j)_{j \in \mathbb{Z}}$ are i.i.d. and square-integrable, we have already checked that the Markov chain associated to $\Pi$ is well-defined.
Recall what means the fact that $\mathbb{Q}_{\beta}$ is carried by $\mathcal{L}$: if $L$ is the set of points corresponding to a $\mathrm{Sine}_{\beta}$ process,
then $L$ is unbounded from above and from below, for $x$ going to infinity, $ \operatorname{Card} (L \cap [0,x] ) = O(x)$ and for all $a, b \in \mathbb{R}$,
$\operatorname{Card} (L \cap [0,x+a] ) - \operatorname{Card} (L \cap [-x+b,0] ) = O(x/\log^2 x)$.
\end{remark}
In order to show the theorem just above, we will use the following results, proven in \cite{NV19}:
\begin{proposition} \label{estimatequadraticCbeta}
Let $L$ be a random set of points in  $\mathbb{R}$,
whose distribution is $\mathbb{Q}_{n, \beta}$ or $\mathbb{Q}_{\beta}$.
Then, there exists $C > 0$, depending on $\beta$ but
not on $n$, such that for all $x > 0$,
$$\mathbb{E} [ \left(\operatorname{Card} (L \cap [0,x])
- x/2 \pi \right)^2] \leq C \log(2+x)$$
and
$$\mathbb{E} [ \left(\operatorname{Card} (L \cap [-x,0])
- x/2 \pi \right)^2] \leq C \log(2+x).$$
\end{proposition}
\begin{proposition}\label{estimatealmostsureCbeta}
Under the previous assumptions, 
for all $\alpha > 1/3$, there exists
 a random
variable $C > 0$,
stochastically dominated by a finite random variable depending only on
$\alpha$ and $\beta$, such that almost surely, for all $x \geq 0$,
$$|\operatorname{Card} (L \cap [0,x])
- x/2 \pi| \leq C (1+x)^{\alpha},$$
and
$$|\operatorname{Card} (L \cap [-x,0])
- x/2 \pi| \leq C (1+x)^{\alpha}.$$
\end{proposition}

\begin{remark}
The periodicity of $L$ implies that
$|\operatorname{Card} (L \cap [0,x])
- x/2 \pi|$ is  almost surely bounded when $x$ varies.
Hence, the result above becomes trivial if one allows
$C$ to depend on $n$.
Moreover, we expect that it remains true for any $\alpha > 0$, and not only for $\alpha > 1/3$.
\end{remark}

{\it Proof of Theorem \ref{betainfinite}. }
Let $\Pi$ be a probability measure which satisfies the assumptions of Theorem \ref{betainfinite}, and for $n \geq 1$, let $\Pi_n$ be a measure satisfying the assumptions
of Theorem \ref{betafinite}, for the same value of $\beta$. We also assume that the law of $h$ is the same under $\Pi_n$ and under $\Pi$ (note that $\Pi_n$ and $\Pi$ are uniquely determined by this
law). By the discussion preceding the statement of Theorem \ref{betainfinite}, it is possible, by using a unique family
$(g_j)_{j \in \mathbb{Z}}$ of i.i.d. gamma variables with parameter $\beta/2$, to construct some random sequences $(\gamma_j)_{j \in \mathbb{Z}}$
and $(\gamma^{n}_j)_{j \in \mathbb{Z}}$ (for all $n \geq 1$) and an independent real-valued random variable $h$, such that the following
holds:
\begin{itemize}
\item $((\gamma_j)_{j \in \mathbb{Z}}, h)$ follows the law $\Pi$.
\item For all $n \geq 1$, $((\gamma^n_j)_{j \in \mathbb{Z}}, h)$ follows the law $\Pi_n$.
\item For all $j \in \mathbb{Z}$, $\gamma^n_j$ tends almost surely to $\gamma_j$ when $n$ goes to infinity.
\end{itemize}
Now, for all $n \geq 1$, let $L_n$ be a point process following the distribution $\mathbb{Q}_{n, \beta}$, and let $L$ be a point process
following $\mathbb{Q}_{\beta}$. We already know that $L_n \in \mathcal{L}$ almost surely. 
From Proposition \ref{estimatealmostsureCbeta} under
$\mathbb{Q}_{\beta}$, we immediately deduce the
weaker estimates
$\operatorname{Card} (L \cap [0,x]) = x/2 \pi + O(x/\log^2 x)$ and
 $\operatorname{Card} (L \cap [-x,0]) = x/2 \pi + O(x/\log^2 x)$ for $x$ going to infinity, which means that $L \in \mathcal{L}$ almost surely: $\mathbb{Q}_{\beta}$ is
 carried by $\mathcal{L}$.

 Moreover, by \cite{bib:KSt09}, the measure $\mathbb{Q}_{n,\beta}$ tends to $\mathbb{Q}_{\beta}$ when $n$ goes to infinity, in the following
 sense: for all functions $f$ from $\mathbb{R}$ to $\mathbb{R}_+$, $C^{\infty}$ and compactly supported, one has
 \begin{equation}
 \sum_{x \in L_n} f(x) \underset{n \rightarrow \infty}{\longrightarrow} \sum_{x \in L} f(x) \label{weakconvergence}
 \end{equation}
 in distribution. By the Skorokhod representation theorem, one can assume that the convergence \eqref{weakconvergence} holds almost surely, and
one can also suppose that $(L_n)_{n \geq 1}$ and $L$ are independent of $(\gamma^n_j)_{n \geq 1, j \in \mathbb{Z}}$, $(\gamma_j)_{j \in \mathbb{Z}}$ and $h$.

For $n \geq 1$, let $(\lambda^n_j)_{j \in \mathbb{Z}}$ be the strictly increasing sequence containing each point of $L_n$, $\lambda^n_0$ being the smallest
nonnegative point, and let $(\lambda_j)_{j \in \mathbb{Z}}$ be the similar sequence associated to $L$. One can check that the convergence \eqref{weakconvergence} and the
fact that $\mathbb{P} [0 \in L] = 0$ imply that for all $j \in \mathbb{Z}$, $\lambda^n_j$ converges almost surely to $\lambda_j$ when $n$ goes to infinity.
Now, for all $c > 0$, $z \in \mathbb{C} \backslash \left(L \cup \left(\bigcup_{n \geq 1} L_n \right) \right)$, let us take the following notation:
$$S_{n,c}(z) := \sum_{j \in \mathbb{Z}} \, \frac{\gamma^n_j}{\lambda^n_j - z} \, \mathds{1}_{|\lambda^n_j| \leq c}, \; \;
S_{c}(z) := \sum_{j \in \mathbb{Z}} \, \frac{\gamma_j}{\lambda_j - z} \, \mathds{1}_{|\lambda_j| \leq c},$$
$$S_N(z) := \lim_{c \rightarrow \infty} S_{N,c} (z), \; \; S(z) := \lim_{c \rightarrow \infty} S_c (z).$$
Almost surely, all the points of $L$ and $L_n$ ($n \geq 1$) are irrational. If this event occurs, then for all $c \in \mathbb{Q}^*_+$, there exists almost surely a finite interval (possibly empty)
$I_c$ such that $|\lambda_j| \leq c$ if and only if $j \in I_c$, and for all $n \geq 1$ large enough, $|\lambda^n_j| \leq c$ if and only if $j \in I_c$. Hence, for all
$c \in \mathbb{Q}^*_+$, $z \in \mathbb{Q}$, one has almost surely
$$S_{n,c}(z) = \sum_{j \in I_c}  \frac{\gamma^n_j}{\lambda^n_j - z}, \; \;
S_{c}(z) := \sum_{j \in I_c} \, \frac{\gamma_j}{\lambda_j - z},$$
if $n$ is large enough. Since $I_c$ is finite, $\gamma^n_j$ tends a.s. to $\gamma_j$, and $\lambda^n_j$ tends a.s. to $\lambda_j$ when $n$ goes to infinity, one deduces that
almost surely, for all $c \in \mathbb{Q}^*_+$, $z \in \mathbb{Q}$,
\begin{equation}
S_{n,c}(z) \underset{n \rightarrow \infty}{\longrightarrow} S_c(z). \label{asconv}
\end{equation}
One the other hand, by \eqref{int}, and by the fact that $c$ and $-c$ are a.s. not in $L$ or in $L_n$, one deduces that almost surely, for all $c \in \mathbb{Q}^*_+$,
 $z \in \mathbb{Q}$ such that $c > 2|z| \vee 1$, and for all $n \geq 1$,
\begin{align*}
S_n (z) - S_{n,c} (z) &  =   \int_{c}^{\infty} \frac{ \Lambda_n([c,  \mu]) - \Lambda_n([-\mu, - c])}{\mu^2} \, d \mu
\\  & +  \int_{c}^{\infty} \left( \frac{ (2 z \mu - z^2) (\Lambda_n([c, \mu]) )}{\mu^2 (\mu-z)^2}
+ \frac{(2 z \mu + z^2) (\Lambda_n([-\mu, -c]) )}{\mu^2 (\mu+z)^2}
 \right) \, d\mu
 \end{align*}
 and
 \begin{align*}
S (z) - S_{c} (z) &  =   \int_{c}^{\infty} \frac{ \Lambda([c,  \mu]) - \Lambda([-\mu, - c])}{\mu^2} \, d \mu
\\  & +  \int_{c}^{\infty} \left( \frac{ (2 z \mu - z^2) (\Lambda([c, \mu]) )}{\mu^2 (\mu-z)^2}
+ \frac{(2 z \mu + z^2) (\Lambda([-\mu, -c]) )}{\mu^2 (\mu+z)^2}
 \right) \, d\mu,
 \end{align*}
where $\Lambda_n:= \sum_{j \in \mathbb{Z}} \gamma^n_j \delta_{\lambda^n_j}$ and $\Lambda := \sum_{j \in \mathbb{Z}} \gamma_j \delta_{\lambda_j}$.
If for any bounded interval $I$, one defines $\Lambda^{(0)}_n (I) := \Lambda_n(I) - \mathbb{E} [ \Lambda_n(I)] $ and
$\Lambda^{(0)} (I) := \Lambda (I) - \mathbb{E} [ \Lambda(I)] $, one has  by \eqref{20z}, the triangle inequality, and the fact that $ \mathbb{E} [ \Lambda_n(I)] $ is
proportional to the Lebesgue measure on $I$:
\begin{align*}
|S_n(z) - S_{n,c}(z)| & \leq \int_c^{\infty} \frac{ |\Lambda^{(0)}_n ([c, \mu])| +   |\Lambda^{(0)}_n ([-\mu, -c])|}{\mu^2} \, d \mu
\\ & + \int_c^{\infty} \frac{20|z| \, d \mu}{\mu^3} \, \left[ C_1 \mu +  |\Lambda^{(0)}_n ([c, \mu])| +   |\Lambda^{(0)}_n ([-\mu, -c])|\right] \\
 & \leq C_2 (1+|z|) \left( \frac{1}{c} +  \int_c^{\infty} \frac{ |\Lambda^{(0)}_n ([c, \mu])| +   |\Lambda^{(0)}_n ([-\mu, -c])|}{\mu^2} \, d \mu \right),
\end{align*}
where $C_1, C_2 > 0$ are universal constants. Since the distribution of $L_n$ is invariant by translation (recall that its points are the rescaled arguments of
the circular beta  ensemble on the unit circle), one has $$\mathbb{E}[ |\Lambda^{(0)}_n ([c, \mu])|] = \mathbb{E}[ |\Lambda^{(0)}_n ([-\mu, -c])|]=
\mathbb{E}[ |\Lambda^{(0)}_n ([0, \mu-c])|]$$
and
$$\mathbb{E}[|S_n(z) - S_{n,c}(z)|] \leq C_3  (1+|z|)  \left( \frac{1}{c} + \int_0^{\infty} \frac{ \mathbb{E}[ |\Lambda^{(0)}_n ([0, \nu])| ]}{(\nu + c)^2} \, d \nu \right),$$
where $C_3 > 0$ is a universal constant. Similarly,
$$\mathbb{E}[|S(z) - S_{c}(z)|] \leq C_3  (1+|z|)  \left( \frac{1}{c} + \int_0^{\infty} \frac{ \mathbb{E}[ |\Lambda^{(0)} ([0, \nu])| ]}{(\nu + c)^2} \, d \nu \right).$$
Now, from Proposition \ref{estimatequadraticCbeta} under $\mathbb{Q}_{n, \beta}$ and $\mathbb{Q}_{\beta}$, one immediately deduces that
\begin{equation}
\int_0^{\infty} \frac{ \mathbb{E}[ |\Lambda^{(0)} ([0, \nu])| ]+ \sup_{n \geq 1} \mathbb{E}[ |\Lambda_n^{(0)} ([0, \nu])| ]}{(1+\nu)^2} \, d\nu < \infty. \label{domination}
\end{equation}
Hence, by dominated convergence, there exists a function $\phi$ from $[1, \infty)$ to $\mathbb{R}^*_+$, tending to zero at infinity, such that
$$\mathbb{E}[|S_n(z) - S_{n,c}(z)|] \leq (1+|z|) \, \phi(c)$$
and
$$\mathbb{E}[|S(z) - S_{c}(z)|] \leq (1+|z|) \, \phi(c).$$
We deduce that for all $c \in \mathbb{Q}^*_+$,
 $z \in \mathbb{Q}$ such that $c > 2|z| \vee 1$, $n \geq 1$ and $\epsilon> 0$,
 \begin{align*}
 \mathbb{P} [ |S(z) - S_{n} (z)| \geq \epsilon] & \leq
 \mathbb{P} [ |S_c(z) - S_{n,c}(z) | \geq \epsilon/3 ]
 +  \mathbb{P} [ |S(z) - S_{c}(z) | \geq \epsilon/3 ] \\ & + \mathbb{P} [ |S_n(z) - S_{n,c}(z)| \geq \epsilon/3 ]
\\ &  \leq  \mathbb{P} [ |S_c(z) - S_{n,c}(z) | \geq \epsilon/3 ] + \frac{6}{\epsilon} \, (1+|z|) \, \phi(c).
 \end{align*}
By the almost sure convergence \eqref{asconv}, which implies the corresponding convergence in probability, one deduces
$$\underset{n \rightarrow \infty}{\lim  \sup}  \; \mathbb{P} [ |S(z) - S_{n} (z)| \geq \epsilon]
\leq  \frac{6}{\epsilon} \, (1+|z|) \, \phi(c).$$
Now, by taking $z \in \mathbb{Q}$ fixed, $c \in \mathbb{Q}$ going to infinity and then $\epsilon \rightarrow 0$, one deduces that for all $z \in \mathbb{Q}$,
$$S_n(z) \underset{n \rightarrow \infty}{\longrightarrow} S(z)$$
in probability. By considering diagonal extraction of subsequences, one deduces that there exists a strictly increasing sequence $(n_k)_{k \geq 1}$ of integers,
such that almost surely,
\begin{equation}
S_{n_k} (z)  \underset{k \rightarrow \infty}{\longrightarrow} S(z) \label{asconv2}
\end{equation}
for all $z \in \mathbb{Q}$.

Now, for all $j \in \mathbb{Z}$, $n \geq 1$, let $\mu^n_j$ (resp. $\mu_j$) be the unique point of $\mathcal{D} (L_n, (\gamma^n_j)_{j \in \mathbb{Z}}, h)$
(resp. $\mathcal{D} (L, (\gamma_j)_{j \in \mathbb{Z}}, h)$) which lies in the interval $(\lambda^n_j, \lambda^n_{j+1})$
(resp. $(\lambda_j, \lambda_{j+1})$). Let us fix $j \in \mathbb{Z}$, $\epsilon > 0$, and let us consider two random rational numbers $q_1$ and $q_2$ such that
almost surely,
$$(\mu_j - \epsilon) \vee \lambda_j < q_1 < \mu_j < q_2 < (\mu_j + \epsilon) \wedge \lambda_{j+1},$$
which implies that
$$S(q_1) < h < S(q_2).$$
By \eqref{asconv2}, one deduces that almost surely, for $k$ large enough,
$$S_{n_k}(q_1) < h < S_{n_k}(q_2),$$
which implies that $\mathcal{D} (L_{n_k}, (\gamma^{n_k}_j)_{j \in \mathbb{Z}}, h)$ has at least one point in the interval $(q_1,q_2)$.
On the other hand, since $\lambda^n_j$ (resp.  $\lambda^n_{j+1}$) tends a.s. to $\lambda_j$ (resp. $\lambda_{j+1}$) when $n$ goes to infinity,
one has almost surely, for $k$ large enough,
$$\lambda^{n_k}_j < q_1 < q_2< \lambda^{n_k}_{j+1}.$$
Hence, $\mathcal{D} (L_{n_k}, (\gamma^{n_k}_j)_{j \in \mathbb{Z}}, h)$ has exactly one point in $(q_1, q_2)$, and this point is necessarily $\mu^{n_k}_j$. One deduces that almost surely,
$|\mu^{n_k}_j - \mu_j| \leq \epsilon$ for $k$ large enough, which implies, by taking $\epsilon \rightarrow 0$, that $\mu^{n_k}_j$ converges
almost surely to $\mu_j$ when $k$ goes to infinity.

Now, let $f$ be a function from $\mathbb{R}$ to $\mathbb{R}_+$, $C^{\infty}$ and compactly supported. Since $L$ is locally finite, there
exists a.s. an integer $j_0 \geq 1$ such that the support of $f$ is included in $(\lambda_{-j_0}, \lambda_{j_0})$, and
then in $(\lambda^{n_k}_{-j_0}, \lambda^{n_k}_{j_0})$ for $k$ large enough, which implies that $f(\mu^{n_k}_j) = f(\mu_j) = 0$ for $|j| > j_0$.
Hence, a.s., there exists $j_0, k_0 \geq 1$, such that for $k \geq k_0$,
$$\sum_{j \in \mathbb{Z}} f(\mu^{n_k}_j) = \sum_{|j| \leq j_0}  f(\mu^{n_k}_j)$$
and
$$\sum_{j \in \mathbb{Z}} f(\mu_j) = \sum_{|j| \leq j_0}  f(\mu_j),$$
which implies that
\begin{equation}
\sum_{j \in \mathbb{Z}} f(\mu^{n_k}_j) \underset{k \rightarrow \infty}{\longrightarrow} \sum_{j \in \mathbb{Z}} f(\mu_j), \label{asconv3}
\end{equation}
since $f(\mu_j^{n_k})$ tends to $f(\mu_j)$ for each $j \in \{-j_0, -j_0+1, \dots, j_0\}$.

The almost sure convergence \eqref{asconv3} holds a fortiori in distribution, which implies
 that the law of $\mathcal{D} (L_{n_k}, (\gamma^{n_k}_j)_{j \in \mathbb{Z}}, h)$ tends to the law of $\mathcal{D} (L, (\gamma_j)_{j \in \mathbb{Z}}, h)$.
On the other hand, by Theorem \ref{betafinite},  $\mathcal{D} (L_{n_k}, (\gamma^{n_k}_j)_{j \in \mathbb{Z}}, h)$ has distribution $\mathbb{Q}_{n_k, \beta}$,
and then $\mathcal{D} (L, (\gamma_j)_{j \in \mathbb{Z}}, h)$ follows the limit of the distribution $\mathbb{Q}_{n_k, \beta}$ for $k$ tending to infinity,
i.e. $\mathbb{Q}_{\beta}$. Theorem \ref{betainfinite} is then proven.

\hfill $\square$

\section{Properties of continuity for the Stieltjes
Markov chain}

In the previous section, we have deduced the convergence of the Markov mechanism associated to $\mathbb{Q}_{n_k, \beta}$
towards the one corresponding to $\mathbb{Q}_{\beta}$ from the convergence of $\mathbb{Q}_{n_k, \beta}$ to
$\mathbb{Q}_{\beta}$ itself, and the convergence of the associated weights.
Later in the paper, we will prove similar results related to the Gaussian ensembles, for which the situation
is more difficult to handle, in particular because of the lack of symmetry of the G$\beta$E at the macroscopic scale, when we rescale around a non-zero point of the bulk. Moreover, we will have to consider several steps of the Markov mechanism at the same time.
 That is why we will need a more general result, giving a property of continuity of the Markov mechanism described above, with respect to its initial data. 

The main results of the present paper concern convergence in distribution of point processes. In this section, we will
assume properties of strong convergence, which can be done with the help of Skorokhod's representation theorem.

The notion of convergence of holomorphic functions usually considered is the uniform convergence on compact sets.
This notion cannot be directly applied to the meromorphic functions involved here, because of the poles on the real line.
That is why we will need an appropriate notion of  uniform convergence of meromorphic functions.

More precisely, we say that a sequence $(f_n)_{n \geq 1}$ of meromorphic functions
on an open set $U \subset \mathbb{C}$  converges uniformly
 to a function $f$ from $U$ to the Riemann sphere $\mathbb{C} \cup  \{\infty\}$ if and only if this  convergence holds for the distance $d$ on $\mathbb{C} \cup  \{\infty\}$,
given by $$d(z_1, z_2) = \frac{|z_2 - z_1|}{ \sqrt{(1+ |z_1|^2)(1+|z_2|^2)}}$$
for $z_1, z_2 \neq \infty$, and extended by continuity at $\infty$ ($d$ corresponds to the
distance of the points on the euclidian sphere, obtained via the inverse stereographic projection).
It is a classical result that the limiting function $f$ should be meromorphic on $U$. One deduces the
following: if a sequence $(f_n)_{n \geq 1}$ of meromorphic functions on $\mathbb{C}$
converges to a function $f$ from $\mathbb{C}$ to $\mathbb{C} \cup  \{\infty\}$, uniformly
on all bounded subsets of $\mathbb{C}$, then $f$ is meromorphic on $\mathbb{C}$.
Morover, the following lemma will be useful:
\begin{lemma} \label{lemmameromorphic}
 Let $(f_n)_{n \geq 1}$ (resp. $(g_n)_{n \geq 1}$) be a sequence of meromorphic functions on an
 open set $U$, uniformly convergent (for the distance $d$) to a function $f$ (resp. $g$),
 necessarily meromorphic. We assume that none of these functions is identically $\infty$
on a connected component of $U$, and that $f$ and $g$ have no common pole.
Then the sequence $(f_n + g_n)_{n \geq 1}$ of meromorphic functions tends uniformly to $f+g$ on all the compact sets of $U$.
\end{lemma}
\begin{proof}
Let $K$ be a compact subset of $U$, let $z_1, z_2, \dots, z_p$ be the poles of $f$ in $K$, and
$z'_1, z'_2, \dots, z'_q$ the poles of $g$ in $K$. There exists a neighborhood
$V$ of $\{z_1,z_2, \dots, z_p\}$ containing no pole of $g$, and a neighborhood
$W$ of $\{z'_1,z'_2, \dots, z'_p\}$ containing no pole of $f$. If $A > 0$ is
fixed, one can assume the following (by restricting $V$ and $W$ if it is needed):
\begin{itemize}
\item The infimum of $|f|$ on $V$ is larger than $2A+1$ and also larger than
 the supremum of $2|g|+1$ on $V$.
 \item The infimum of $|g|$ on $W$ is larger than $2A+1$ and also larger than
 the supremum of $2|f|+1$ on $W$.
\end{itemize}
By the assumption of uniform convergence, we deduce, for $n$ large enough:
\begin{itemize}
\item The infimum of $|f_n|$ on $V$ is larger than $2A$ and also larger than
 the supremum of $2|g_n|$ on $V$.
 \item The infimum of $|g_n|$ on $W$ is larger than $2A$ and also larger than
 the supremum of $2|f_n|$ on $W$.
\end{itemize}
Now, for all $z \in V$ and $n$ large enough, one has
$$|f_n(z) + g_n(z)| \geq |f_n(z)| - |g_n(z)| \geq |f_n(z)| - \frac{ |f_n(z)|}{2}
= \frac{ |f_n(z)|}{2} \geq A.$$
and also
$$|f(z) + g(z)| \geq A,$$
which implies
$$d(f_n(z) + g_n(z), f(z) + g(z)) \leq 2/A.$$
Similarly, this inequality is true for $z \in W$.
Moreover, there exists a compact set $L \subset K$, containing no pole of $f$ or $g$,
and such that $K$ is included in $L \cup V \cup W$.
Since the meromorphic functions $f$ and $g$ have no pole on the compact set $L$,
they are bounded on this set. Since $(f_n)_{n \geq 1}$ (resp. $(g_n)_{n \geq 1}$) converges to $f$ (resp. $g$) on $L$, uniformly for the distance $d$, and $(f_n)_{n \geq 1}$ (resp. $(g_n)_{n \geq 1}$)
is uniformly bounded, the uniform convergence holds in fact for the usual distance.
Hence, $(f_n + g_n)_{n \geq 1}$ tends uniformly to $f+g$ on $L$ for the usual distance, and a
fortiori for $d$: by using the previous bounded obtained in $V$ and $W$, one deduces,
since $L$, $V$ and $W$ cover $K$:
$$\underset{n \rightarrow \infty} {\lim \sup} \, \underset{z \in K}{\sup}
\, d(f_n(z) + g_n(z), f(z) + g(z)) \leq 2/A.$$
Since we can choose $A > 0$ arbitrarily, we are done.
\end{proof}
From this lemma, we deduce the following statement
\begin{lemma} \label{lemmarational}
Let $p \geq 1$, and let $(\lambda_k)_{1 \leq k \leq p}$, $(\lambda_{n,k})_{n \geq 1, 1 \leq k \leq p}$, $(\gamma_k)_{1 \leq k \leq p}$, $(\gamma_{n,k})_{n \geq 1, 1 \leq k \leq p}$ be
some complex numbers such that all the $\lambda_k$'s are distincts, all the
$\gamma_k$'s are nonzero, and for all $k \in \{1, \dots, p\}$,
$$\lambda_{n,k} \underset{n \rightarrow \infty}{\longrightarrow} \lambda_k$$
and
$$\gamma_{n,k} \underset{n \rightarrow \infty}{\longrightarrow} \gamma_k.$$
Then, one has, for $n$ going to infinity, the convergence of the rational function
$$z \mapsto \sum_{k=1}^p \frac{\gamma_{n,k}}{\lambda_{n,k} - z}$$
towards the function
$$z \mapsto \sum_{k=1}^p \frac{\gamma_{k}}{\lambda_{k} - z},$$
uniformly on all the compact sets, for the distance $d$.
\end{lemma}
\begin{proof}
Let us first prove the result for $p=1$, which is implied by the following
convergence
$$\frac{\gamma_{n,1}}{\lambda_{n,1} - z } \underset{n \rightarrow \infty}{\longrightarrow}
\frac{\gamma_{1}}{\lambda_{1} - z },$$
uniformly on $\mathbb{C}$ for the distance $d$.
Let us fix $\epsilon > 0$. For $n$ large enough, we have
$|\lambda_{n,1} - \lambda_{1}| \leq \epsilon$ and
$|\gamma_{n,1} - \gamma_{1}| \leq |\gamma_1|/2$.
If these conditions are satisfied and if $|\lambda_{1} - z| \leq 2 \epsilon$,
then
$$\left|\frac{\gamma_{1}}{\lambda_{1} - z } \right|  \geq \frac{|\gamma_1|}{2 \epsilon}$$
and
$$\left|\frac{\gamma_{n,1}}{\lambda_{n,1} - z } \right|  \geq \frac{|\gamma_1|}{6 \epsilon},$$
since $|\gamma_{n,1}| \geq |\gamma_1|/2$ and
$$|\lambda_{n,1} - z| \leq |\lambda_{1} - z| + |\lambda_{n,1} - \lambda_{1}|\leq 3 \epsilon.$$
Hence, there exists $n_0 \geq 1$, independent of $z$ satisfying
$|\lambda_{1} - z| \leq 2 \epsilon$, such that for $n \geq n_0$,
$$d\left( \frac{\gamma_{1}}{\lambda_{1} - z },
\frac{\gamma_{n,1}}{\lambda_{n,1} - z } \right)
\leq d\left( \frac{\gamma_{1}}{\lambda_{1} - z },
\infty \right) +
d\left( \infty,
\frac{\gamma_{n,1}}{\lambda_{n,1} - z } \right)
\leq \frac{2 \epsilon}{|\gamma_1|} + \frac{6 \epsilon}{|\gamma_1|}  =  \frac{8 \epsilon}{|\gamma_1|}.$$
Similarly, there exists $n_1 \geq 1$ such that for all $n \geq n_1$ and for
all $z$ satisfying $|\lambda_{1} - z| \geq 2 \epsilon$, one has:
$$|\lambda_{n,1} - z| \geq |\lambda_{1} - z| - |\lambda_{n,1} - \lambda_{1}| \geq \epsilon.$$
This implies:
\begin{align*}
 \left| \frac{\gamma_{1}}{\lambda_{1} - z } -  \frac{\gamma_{n,1}}{\lambda_{n,1} - z } \right|
 & \leq \left| \frac{\gamma_{1} - \gamma_{n,1} }{\lambda_{n,1} - z } \right|
 + |\gamma_1| \, \left| \frac{1}{\lambda_{1} - z} - \frac{1}{\lambda_{n,1} - z} \right|
 \\ & \leq \frac{|\gamma_{1} - \gamma_{n,1} |}{\epsilon}
 + |\gamma_1| \, \frac{|\lambda_{1} - \lambda_{n,1}|}{(2\epsilon)(\epsilon)}.
\end{align*}
Since this quantity does not depend on $z$ and tends to zero at infinity,
we deduce
$$\sup_{z \in \mathbb{C}, |\lambda_{1} - z| \geq 2 \epsilon}
d\left( \frac{\gamma_{1}}{\lambda_{1} - z },
\frac{\gamma_{n,1}}{\lambda_{n,1} - z } \right) \underset{n \rightarrow \infty}
{\longrightarrow} 0.$$
Since we know that
$$\underset{n \rightarrow \infty}{\lim \, \sup}
\sup_{z \in \mathbb{C}, |\lambda_{1} - z| \leq 2 \epsilon}
d\left( \frac{\gamma_{1}}{\lambda_{1} - z },
\frac{\gamma_{n,1}}{\lambda_{n,1} - z } \right)\leq \frac{8 \epsilon}{|\gamma_1|},$$
we get
$$\underset{n \rightarrow \infty}{\lim \, \sup} \,
\sup_{z \in \mathbb{C}} \,
d\left( \frac{\gamma_{1}}{\lambda_{1} - z },
\frac{\gamma_{n,1}}{\lambda_{n,1} - z } \right)\leq \frac{8 \epsilon}{|\gamma_1|}.$$
Now, $\epsilon > 0$ can be arbitrarily chosen, and then the lemma is proven for $p=1$.
For $p \geq 2$, let us deduce the result of the lemma, assuming that it is
satisfied when $p$ is replaced by $p-1$.
We define the meromorphic functions $(f_n)_{n \geq 1}$, $f$, $(g_n)_{n \geq 1}$, $g$
by the formulas:
$$f_n(z) = \sum_{k=1}^{p-1} \frac{\gamma_{n,k}}{\lambda_{n,k} - z},$$
$$f(z) = \sum_{k=1}^{p-1} \frac{\gamma_{k}}{\lambda_{k} - z},$$
$$g_n(z) = \frac{\gamma_{n,p}}{\lambda_{n,p} - z},$$
$$g(z) = \frac{\gamma_{p}}{\lambda_{p} - z}.$$
Let $A > 0$. By the induction hypothesis, we know that $f_n$ converges to $f$ when
$n$ goes to infinity, uniformly on the set $\{z \in \mathbb{C}, |z| < 2A\}$ and for
the distance $d$. Similarly, by the case $p=1$ proven above, $g_n$ converges to $g$, uniformly on the same set (in fact, uniformly on $\mathbb{C}$) and for the same distance. Moreover,
the functions $f$ and $g$ have no common pole, since the numbers $(\lambda_k)_{1 \leq k \leq p}$
are all distinct. We can then apply Lemma \ref{lemmameromorphic} and deduce
that $f_n+ g_n$ converges to $f + g$, uniformly on any compact set of
 $\{z \in \mathbb{C}, |z| < 2A\}$, for example $\{z \in \mathbb{C}, |z| \leq A\}$,
 and for the distance $d$. Since $A > 0$ can be arbitrarily chosen, we are done.

\end{proof}

We have now the ingredients needed to state the main
result of this section. In this theorem, we deal with finite and infinite sequences together. So we will think of $k\mapsto \lambda_k$ as a function from $\mathbb Z\to \mathbb R\cup \{\emptyset\}$, with the convention that summation and other operations are only considered over the values that are different from $\emptyset$. We will also assume that the value $\emptyset$ is taken exactly on the complement of an interval of
$\mathbb Z$.

The statement of the following result is long and
technical, but as we will see in the next section, it
will be adapted to the problem we are interested in.

\begin{theorem} \label{topologytheorem}
Let $(\Xi_n)_{n \geq 1}$ be a sequence of discrete simple point measures on $\mathbb R$ (i.e. sums of
Dirac masses at a locally finite set of points), converging to a simple point measure $\Xi$, locally weakly:
 \begin{equation}\label{e:lambda1}
  \Xi_n \longrightarrow \Xi.
\end{equation}
Let $L_n$ denote the support of $\Xi_n$, and $L$ the support of $\Xi$.
We suppose that there exists $\alpha\in(0,1)$, a family $(\tau_\ell)_{\ell \geq 0}$
of elements of $\mathbb{R}_+^*$, with $\tau_\ell\to 0$ as $\ell\to\infty$, such that for all $n \geq 1$, $\ell \geq 1$, we have
\begin{equation}\label{e:lambdasquared}
  \int_{\mathbb R}\frac{\mathbf 1(|\lambda|>\ell)}{|\lambda|^{1+\alpha}}\,d\Xi_n(\lambda)
  \le \tau_\ell
\end{equation}
Moreover, assume that the limits
\begin{equation}\label{e:lambdasquared2}
h_{n,\ell} = \lim_{\ell'\to\infty}  \int_{\mathbb R}\frac{\mathbf 1(\ell<|\lambda|<\ell')}{\lambda}\,d\Xi_n(\lambda)
\end{equation}
exist, and so does the similar limit $h_\ell$ defined in terms of $\Xi$. Assume further that
for some $h\in \mathbb R$, the
following equalities are well-defined and satisfied:
\begin{equation}\label{e:hlim}
\lim_{\ell\to \infty}\lim_{n\to \infty} h_{n,\ell} =h, \qquad  \lim_{\ell\to \infty}h_{\ell} =0,
\end{equation}
when the limits are restricted to
the condition: $\ell \notin L$ and $ -\ell \notin L$.

Further, let $(\gamma_{n,k})_{k\in \mathbb Z}$ be a strictly positive sequence.
 Suppose it satisfies
\begin{equation}
  \gamma_{n,k} \to \gamma_k > 0 \label{e:gammalim}
\end{equation}
for each $k$, as $n\to\infty$. Also for some $\gammab,c>0$ and all $n,m \geq 1$, we assume
\begin{eqnarray}\label{e:gammatight}\notag
  \left|\sum_{k=0}^{m-1} \gamma_{n,k} - \gammab m\right| &\le& c m^{\alpha'}\\
  \left|\sum_{k=-m}^{-1} \gamma_{n,k} - \gammab m\right| &\le& c m^{\alpha'}
\end{eqnarray}
with $0<(1+\alpha)\alpha'<1$. Let $\lambda^*$ be a point outside $L$, and consider the weighted version $\Lambda$ of $\Xi$ where the $k$-th point after $\lambda^*$ (for $k \leq 0$,
the $(1-k)$-th point before $\lambda^*$) has weight $\gamma_k$. For $n$ large enough, one has also $\lambda^* \notin L_n$: define $\Lambda_n$ similarly.
Then the limit
$$
S_n(z)=\lim_{\ell\to\infty} \int_{[-\ell,\ell]}
\frac{1}{\lambda-z}\,d\Lambda_n(\lambda)
$$
exists for all $z\notin L_n$, is meromorphic with simple poles at $L_n$, and converges, uniformly on compacts with respect to the distance $d$ on the Riemann sphere $\mathbb{C} \cup \{\infty\}$, to $S(z)+\bar{\gamma}h$, where $S$ is a meromorphic function with simple poles at $L$, such that for all $z \notin L$,
$$
S(z)=\lim_{\ell\to\infty} \int_{[-\ell,\ell]}
\frac{1}{\lambda-z}\,d\Lambda(\lambda).
$$
 Moreover, for every $h'\in \mathbb R$,
the sum of delta masses $\Xi'_n$ at $S_n^{-1}(h'+\bar{\gamma}h)$ converges locally weakly to the sum of delta masses $\Xi'$ at $S^{-1}(h')$, and $(\Xi'_{n})_{n \geq 1}$, $\Xi'$ satisfy the assumptions from \eqref{e:lambda1} to
\eqref{e:hlim}.
\end{theorem}

\begin{proof}
Let $n \geq 1$, large enough in order to ensure that $\lambda^* \notin L_n$, $\ell > \ell_0 > 1$, and let $z$ be a complex number with modulus smaller than $\ell_0/2$.
Let $k_{n, \ell_0}$ be the smallest index $k$ (if it exists) such that
$\lambda_{n,k} > \ell_0$, where $\lambda_{n,k}$ is (if it exists) the $k$-th point of $L_n$ after $\lambda^*$. Similarly, let $K_{n, \ell}-1$ be the largest index $k$ (if it exists) such that
$\lambda_{n,k} \leq \ell$.
If for $m \in \mathbb{Z}$, $$\Delta_{n,m} := \mathds{1}_{m \geq 0} \sum_{k=0}^{m-1} \gamma_{n,k}
- \mathds{1}_{m < 0} \sum_{k=m}^{-1} \gamma_{n,k} - \bar{\gamma} \,m,$$
then, in the case where $k_{n, \ell_0}$ and $K_{n, \ell}$ are well-defined and
$k_{n, \ell_0} < K_{n, \ell}$:
\begin{align*}
 \int_{(\ell_0,\ell]}
\frac{1}{\lambda-z}\,d\Lambda_n(\lambda)
& = \sum_{k_{n,\ell_0} \leq k < K_{n,\ell} } \frac{ \gamma_{n,k}}{\lambda_{n,k} - z}
 = \bar{\gamma} \sum_{k_{n,\ell_0} \leq k < K_{n,\ell}} \frac{1}{\lambda_{n,k} - z}
 \\ &  +  \sum_{k_{n,\ell_0} \leq k < K_{n,\ell}} \frac{\Delta_{n,k+1} - \Delta_{n,k}}{\lambda_{n,k} - z}
 =  \bar{\gamma} \left( \sum_{k_{n,\ell_0} \leq k < K_{n,\ell}} \frac{1}{\lambda_{n,k} - z} \right) +
 \frac{\Delta_{n,K_{n,\ell}}}{ \lambda_{n, K_{n,\ell}-1} - z } \\ & - \frac{\Delta_{n,
k_{n,\ell_0}}}{ \lambda_{n, k_{n,\ell_0}} - z }
 + \sum_{k_{n,\ell_0} +1 \leq k < K_{n,\ell} } \Delta_{n,k} \left(\frac{1}{\lambda_{n,k-1} -z} -
\frac{1}{\lambda_{n,k} -z} \right),
\end{align*}
which implies
\begin{align*}
 \int_{(\ell_0,\ell]}
\frac{1}{\lambda-z}\,d\Lambda_n(\lambda) & -  \bar{\gamma} \int_{(\ell_0, \ell]}  \, \frac{d\Xi_n(\lambda)}{\lambda}  =
\bar{\gamma} z \left( \sum_{k_{n,\ell_0} \leq k < K_{n,\ell}} \frac{1}{\lambda_{n,k} (\lambda_{n,k} - z)} \right)
+ \frac{\Delta_{n,K_{n,\ell}}}{ \lambda_{n, K_{n,\ell} - 1} - z } \\ &
 - \frac{\Delta_{n,k_{n,\ell_0}}}{ \lambda_{n, k_{n,\ell_0}} - z }
 + \sum_{k_{n,\ell_0} +1 \leq k < K_{n,\ell} } \Delta_{n,k} \left(\frac{\lambda_{n,k}
 - \lambda_{n,k-1}}{(\lambda_{n,k-1} -z)
 (\lambda_{n,k} -z)} \right).
\end{align*}
Note that in case where $k_{n,\ell_0}$ or $K_{n,\ell}$ is not well-defined, and in case where
$k_{n,\ell_0} \geq K_{n,\ell}$, the left-hand side is zero, since $L_n$ has no
point in the interval $(\ell_0, \ell]$.
Let us now check that for $\ell$ going to infinity, this quantity converges,
 uniformly in $\{z \in \mathbb{C}, |z| <
\ell_0/2\}$, to the function $T_{n, \ell_0}$, holomorphic on this open set, and given by
 \begin{align}
T_{n, \ell_0}(z) & =
\bar{\gamma} z \left( \sum_{k \geq k_{n,\ell_0} } \frac{1}{\lambda_{n,k} (\lambda_{n,k} - z)} \right)
 + \frac{\Delta_{n,K_{n,\infty}}}{ \lambda_{n, K_{n,\infty}-1} - z } - \frac{\Delta_{n,k_{n,\ell_0}}}{ \lambda_{n, k_{n,\ell_0}} - z } \nonumber  \\ &
 + \sum_{ k \geq  k_{n,\ell_0}+1} \Delta_{n,k} \left(\frac{\lambda_{n,k}
 - \lambda_{n,k-1}}{(\lambda_{n,k-1} -z)
 (\lambda_{n,k} -z)} \right), \label{Tnz}
 \end{align}
 if $L_n$ has at least one point in $(\ell_0, \infty)$, and $T_{n, \ell_0} (z) = 0$
 otherwise. In the formula above, $K_{n, \infty}-1$ denotes the index of the largest point
 of $L_n$ if this set is bounded from above, and otherwise,
 $$\frac{\Delta_{n,K_{n,\infty}}}{ \lambda_{n, K_{n,\infty}-1} - z } := 0.$$
 In order to prove this convergence, it is sufficient to check,
 in case where
 $L_n \cap (\ell_0, \infty) \neq \emptyset$,
 the uniform convergence
 $$
 \frac{\Delta_{n,K_{n,\ell}}}{ \lambda_{n, K_{n,\ell} - 1} - z }
 \underset{\ell \rightarrow \infty}{\longrightarrow}
  \frac{\Delta_{n,K_{n,\infty}}}{ \lambda_{n, K_{n,\infty} - 1} - z },$$
  for $|z| \leq \ell_0/2$, and the fact that
 \begin{equation}
 \sup_{z \in \mathbb{C}, |z| \leq \ell_0/2} \left( \sum_{k \geq k_{n,\ell_0} }  \frac{1}{|\lambda_{n,k}||\lambda_{n,k}
  - z|}+ \sum_{ k \geq  k_{n,\ell_0}+1} |\Delta_{n,k}| \left(\frac{|\lambda_{n,k}
 - \lambda_{n,k-1}|}{|\lambda_{n,k-1} -z| \,
 |\lambda_{n,k} -z|} \right) \right)< \infty. \label{deltafinite}
 \end{equation}
 The first statement is  immediate if $L_n$ is bounded from above. If $L_n$ is unbounded from
 above,
 let us remark that for $k \geq k_{n, \ell_0}$, $|z| \leq \ell_0/2$, one
has $|\lambda_{n,k} - z| \geq \lambda_{n,k}/2$, and then it is sufficient to
show:
\begin{equation}
\frac{\Delta_{n,K_{n,\ell}}}{ \lambda_{n, K_{n,\ell} - 1}}  \underset{\ell \rightarrow \infty}{\longrightarrow} 0. \label{deltak1}
\end{equation}
Similarly, the statement \eqref{deltafinite} is implied by:
\begin{equation}
\sum_{k \geq k_{n,\ell_0} } \frac{1}{\lambda_{n,k}^2}+ \sum_{ k \geq  k_{n,\ell_0}+1}  \frac{
|\Delta_{n,k}| (\lambda_{n,k}
 - \lambda_{n,k-1})}{\lambda_{n,k-1} \,
 \lambda_{n,k}} < \infty. \label{deltak2}
\end{equation}
In order to prove \eqref{deltak1}, let us first use the majorization
 \eqref{e:lambdasquared}, which implies, for all $\ell > 2$,
 $$\Xi_n ([2, \ell]) \leq
 \ell^{1 + \alpha} \, \int_{2}^{\ell} \frac{d\Xi_n(\lambda)}{\lambda^{1 + \alpha}}
 \leq  \ell^{1 + \alpha}  \, \int_{\mathbb R}\frac{\mathbf 1(|\lambda|>1)}{|\lambda|^{1+\alpha}}\,d\Xi_n(\lambda) \leq \tau_1 \,  \ell^{1 + \alpha},$$
 and then
 $$\Xi_n ([\lambda^*, \ell]) \leq \tau \, \ell^{1+ \alpha}$$
 where
 \begin{equation}
 \tau := \tau_1 + \Xi_n ([\lambda^* \wedge 2, 2]). \label{tau}
 \end{equation}
We deduce, for $k \geq 1$ large enough in order to insure that $\lambda_{n,k} > 2$,
$$k = \Xi_n ([\lambda^*, \lambda_{n,k}]) \leq \tau \, \lambda_{n,k}^{1+\alpha}$$
and then
\begin{equation}
\lambda_{n,k} \geq (k/ \tau)^{1/(1+\alpha)}. \label{minorizationlambda}
\end{equation}
By using \eqref{e:gammatight}, this inequality implies:
$$\frac{\Delta_{n,k+1}}{ \lambda_{n,k}} \leq c (k+1)^{\alpha'} (k/\tau)^{-1/(1+\alpha)},$$
which tends to zero when $k$ goes to infinity, since $\alpha' < 1/(1+\alpha)$ by assumption.
Therefore, we have \eqref{deltak1}.

Moreover, the left-hand side of \eqref{deltak2} is given by
\begin{align*}
 & \int_{\mathbb R}  \frac{\mathbf 1(|\lambda|>\ell_0)}{|\lambda|^{2}}\,d\Xi_n(\lambda)
+ \sum_{ k \geq  k_{n,\ell_0}+1}
|\Delta_{n,k}|  \, \left( \frac{1}{\lambda_{n,k-1}} - \frac{1}{\lambda_{n,k}} \right)
\\ & \leq \int_{\mathbb R}\frac{\mathbf 1(|\lambda|>\ell_0)}{|\lambda|^{1+\alpha}}\,d\Xi_n(\lambda)
+ c \, \sum_{ k \geq  k_{n,\ell_0}+1} |k|^{\alpha'} \, \left( \frac{1}{\lambda_{n,k-1}} - \frac{1}{\lambda_{n,k}} \right)
\\ & \leq \tau_{\ell_0} + c \left( \frac{|k_{n,\ell_0}+1|^{\alpha'}}{\lambda_{n,k_{n,\ell_0}}}
+ \sum_{ k \geq  k_{n,\ell_0}+1} \frac{(|k+1|^{\alpha'} - |k|^{\alpha'})}{\lambda_{n,k}} \right).
\end{align*}
If $L_n$ is bounded from above, the finiteness of this quantity is obvious. Otherwise, we know
that for $k$ large enough, $(|k+1|^{\alpha'} - |k|^{\alpha'})$ is bounded by a constant
times $k^{\alpha' - 1}$, and $\lambda_{k, n}$ dominates $k^{1/(1+\alpha)}$. Hence, it is sufficient
to check the finiteness of the following expression:
$$ \sum_{k = 1}^{\infty}  k^{\alpha' -1} k^{-1/(1+\alpha)},$$
which is satisfied since by assumption,
$$\alpha' - 1 - \frac{1}{1 + \alpha} < -1.$$
We have now proven:
\begin{equation} \int_{(\ell_0,\ell]}
\frac{1}{\lambda-z}\,d\Lambda_n(\lambda)  -  \bar{\gamma} \int_{(\ell_0, \ell]}  \, \frac{d\Xi_n(\lambda)}{\lambda} \underset{\ell \rightarrow \infty}{\longrightarrow}
T_{n, \ell_0}(z),  \label{uniformTnz}
\end{equation}
uniformly on the set $\{z \in \mathbb{C}, |z| < \ell_0/2\}$, where the holomorphic function
$T_{n, \ell_0}$ is given by the formula \eqref{Tnz}.

Similarly, there exists an holomorphic function $U_{n, \ell_0}$ on
$\{z \in \mathbb{C}, |z| < \ell_0/2\}$, such that uniformly on this set,
 \begin{equation} \int_{[-\ell, -\ell_0)}
\frac{1}{\lambda-z}\,d\Lambda_n(\lambda)  -  \bar{\gamma} \int_{[-\ell, -\ell_0)}  \, \frac{d\Xi_n(\lambda)}{\lambda} \underset{\ell \rightarrow \infty}{\longrightarrow}
U_{n, \ell_0}(z).  \label{uniformUnz}
\end{equation}
The function $U_{n, \ell_0}$ can be explicitly described by a formula similar to \eqref{Tnz} (we omit
the detail of this formula).
By combining \eqref{e:lambdasquared2}, \eqref{uniformTnz} and \eqref{uniformUnz}, one deduces
the following uniform convergence on $\{z \in \mathbb{C}, |z| < \ell_0/2\}$:
 $$\int_{[-\ell, \ell] \backslash [-\ell_0, \ell_0]}
\frac{1}{\lambda-z}\,d\Lambda_n(\lambda)  \underset{\ell \rightarrow \infty}{\longrightarrow}
T_{n, \ell_0}(z) + U_{n, \ell_0}(z) + \bar{\gamma} h_{n, \ell_0}.$$
One deduces, by using Lemma \ref{lemmameromorphic}, that
$$\int_{[-\ell, \ell]}
\frac{1}{\lambda-z}\,d\Lambda_n(\lambda)
\underset{\ell \rightarrow \infty}{\longrightarrow}
T_{n, \ell_0}(z) + U_{n, \ell_0}(z) + \bar{\gamma} h_{n, \ell_0}
+ \int_{[-\ell_0, \ell_0]}
\frac{1}{\lambda-z}\,d\Lambda_n(\lambda) =: S_{n, \ell_0} (z),$$
uniformly on any compact subset of $\{z \in \mathbb{C}, |z| < \ell_0/2\}$, for the
distance $d$ on the Riemann sphere. One checks immediately that the poles of $S_{n, \ell_0}$ with modulus smaller than or equal to $\ell_0/2$ are exactly the points of $L_n$ satisfying the
same condition. Moreover, the convergence just above implies that for $\ell_1 > \ell_0 > 1$,
the meromorphic functions $S_{n, \ell_0}$ and $S_{n, \ell_1}$ coincide on
$\{z \in \mathbb{C}, |z| < \ell_0/2\}$: hence, there exists a meromorphic function
$S_n$ on $\mathbb{C}$, such that for all $\ell_0 > 1$, the restriction
of $S_n$ to $\{z \in \mathbb{C}, |z| < \ell_0/2\}$ is equal to $S_{n, \ell_0}$.
The poles of $S_n$ are exactly the points of $L_n$, and one has, uniformly on
all compact sets of $\mathbb{C}$ and for the distance $d$,
$$\int_{[-\ell, \ell]}
\frac{1}{\lambda-z}\,d\Lambda_n(\lambda)
\underset{\ell \rightarrow \infty}{\longrightarrow} S_n(z).$$
In particular, the convergence holds pointwise for all $z \notin L_n$.

In an exactly similar way, one can prove that uniformly on compact sets of $\mathbb{C}$,
for the distance $d$,
$$\int_{[-\ell, \ell]}
\frac{1}{\lambda-z}\,d\Lambda(\lambda)
\underset{\ell \rightarrow \infty}{\longrightarrow} S(z)$$
where for all $\ell_0 > 1$,
$$S(z) := T_{\ell_0}(z) + U_{\ell_0}(z) + \bar{\gamma} h_{\ell_0}
+ \int_{[-\ell_0, \ell_0]}
\frac{1}{\lambda-z}\,d\Lambda(\lambda),$$
on the set $\{z \in \mathbb{C}, |z|< \ell_0/2\}$, $T_{\ell_0}$ and $U_{\ell_0}$ being
defined by the same formulas as $T_{n,\ell_0}$ and $U_{n,\ell_0}$, except than one removes all the
indices $n$. In order to show this convergence, it is sufficient to check that the
assumptions \eqref{e:lambdasquared} and  \eqref{e:gammatight} are satisfied if the
indices $n$ are removed. For \eqref{e:gammatight}, it is an immediate consequence
 of the convergence \eqref{e:gammalim}, since the constant $c$ does not depend on $n$.
For \eqref{e:lambdasquared}, let us first observe that for all $\ell > 1$, and for
any continuous function $\Phi$ with compact support, such that for all $\lambda \in \mathbb{R}$,
$$\Phi(\lambda) \leq \frac{\mathbf{1}(|\lambda| > \ell)}{|\lambda|^{1+\alpha}},$$
one has for all $n \geq 1$,
$$ \int_{\mathbb R} \Phi(\lambda) d\Xi_n(\lambda)  \leq \int_{\mathbb R}\frac{\mathbf 1(|\lambda|>\ell)}{|\lambda|^{1+\alpha}}\,d\Xi_n(\lambda)
  \le \tau_\ell.$$
Since $\Xi_n$ converges weakly to $\Xi$ when $n$ goes to infinity, one deduces:
$$\int_{\mathbb R} \Phi(\lambda) d\Xi(\lambda) \leq \tau_{\ell}.$$
By taking $\Phi$ increasing to
$$\lambda \mapsto \mathbf{1}(|\lambda| > \ell)/ |\lambda|^{1+\alpha},$$
one obtains
$$\int_{\mathbb R}\frac{\mathbf 1(|\lambda|>\ell)}{|\lambda|^{1+\alpha}}\,d\Xi(\lambda)
  \le \tau_\ell,$$
  i.e. the equivalent of  \eqref{e:lambdasquared} for the measure $\Xi$.

  Once the existence of the functions $S_n$ and $S$ is ensured, it remains to
  prove the convergence of $S_n$ towards $S+ \bar{\gamma} h$, uniformly on compact sets for the distance $d$.
  In order to check this convergence, it is sufficient to prove that there exists
  $\ell_0 > 1$ arbitrarily large, such that uniformly on any compact set
  of $\{z \in \mathbb{C}, |z| < \ell_0/2\}$,
  \begin{align*}
 & T_{n, \ell_0}(z) + U_{n, \ell_0}(z) + \bar{\gamma} h_{n, \ell_0}
+ \int_{[-\ell_0, \ell_0]}
\frac{1}{\lambda-z}\,d\Lambda_n(\lambda)
\\ & \underset{n \rightarrow \infty}{\longrightarrow} T_{\ell_0}(z) + U_{\ell_0}(z) + \bar{\gamma} (h_{ \ell_0} + h )
+ \int_{[-\ell_0, \ell_0]}
\frac{1}{\lambda-z}\,d\Lambda(\lambda).
 \end{align*}
  In fact, we will prove this convergence for any $\ell_0 >2$ such that $\ell_0$ and
  $-\ell_0$ are not in $L$, and then not in $L_n$ for $n$ large enough.
  By Lemma \ref{lemmameromorphic}, it is sufficient to check for such an $\ell_0$:
  \begin{equation}
   h_{n, \ell_0} \underset{n \rightarrow \infty}{\longrightarrow} h_{\ell_0} + h, \label{convergenceh}
  \end{equation}
 \begin{equation}
  \int_{[-\ell_0, \ell_0]}
\frac{1}{\lambda-z}\,d\Lambda_n(\lambda)
\underset{n \rightarrow \infty}{\longrightarrow}
 \int_{[-\ell_0, \ell_0]}
\frac{1}{\lambda-z}\,d\Lambda(\lambda), \label{convergencegamma}
\end{equation}
uniformly on $\{z \in \mathbb{C}, |z| < \ell_0/2\}$ for the distance $d$,
\begin{equation}
T_{n, \ell_0}(z) \underset{n \rightarrow \infty}{\longrightarrow} T_{\ell_0}(z), \label{convergenceT}
\end{equation}
uniformly on $\{z \in \mathbb{C}, |z| < \ell_0/2\}$, and
\begin{equation}
U_{n, \ell_0}(z) \underset{n \rightarrow \infty}{\longrightarrow} U_{\ell_0}(z), \label{convergenceU}
\end{equation}
also uniformly on $\{z \in \mathbb{C}, |z| < \ell_0/2\}$. Since the
proof of \eqref{convergenceU} is exactly similar to the proof of \eqref{convergenceT},
we will omit it and we will then show successively \eqref{convergenceh},  \eqref{convergencegamma}
and \eqref{convergenceT}.

Let us first prove \eqref{convergenceh}. For all $\ell_1 > \ell_0 $ such that
$-\ell_1$ and $\ell_1$ are not in $L$, one has
$$h_{n, \ell_0} - h_{n, \ell_1} =  \int_{\mathbb R}\frac{\mathbf 1(\ell_0<|\lambda| \leq \ell_1)}{\lambda}\,d\Xi_n(\lambda)$$
and
$$h_{\ell_0} - h_{\ell_1} =  \int_{\mathbb R}\frac{\mathbf 1(\ell_0<|\lambda| \leq \ell_1)}{\lambda}\,d\Xi(\lambda).$$
Now, $-\ell_1, -\ell_0, \ell_0, \ell_1$ are not in the support of $\Xi$, and since
$\Xi$ is a discrete measure, there is a neighborhood of $\{-\ell_1, -\ell_0, \ell_0, \ell_1\}$
which does not charge $\Xi$. One deduces that there exist two functions
$\Phi$ and $\Psi$ from $\mathbb{R}$ to $\mathbb{R}_+$, continuous with compact support,
such that for all $\lambda \in \mathbb{R}$,
$$\Phi(\lambda) \leq \frac{\mathbf 1(\ell_0<|\lambda| \leq \ell_1)}{\lambda}
\leq \Psi(\lambda)$$
and
$$\int_{\mathbb R} \Phi(\lambda) \, d\Xi(\lambda) =  h_{\ell_0} - h_{\ell_1}
= \int_{\mathbb R} \Psi(\lambda) \, d\Xi(\lambda).$$
Since $\Xi_n$ tends weakly to $\Xi$ when $n$ goes to infinity, one deduces
that
$$\int_{\mathbb R} \Phi(\lambda) \, d\Xi_n(\lambda) \underset{n \rightarrow \infty}
{\longrightarrow} \int_{\mathbb R} \Phi(\lambda) \, d\Xi(\lambda) =  h_{\ell_0} - h_{\ell_1} $$
and similarly,
$$\int_{\mathbb R} \Psi(\lambda) \, d\Xi_n(\lambda) \underset{n \rightarrow \infty}
{\longrightarrow} h_{\ell_0} - h_{\ell_1}.$$
By the squeeze theorem, one deduces
$$ h_{n, \ell_0} - h_{n, \ell_1} =  \int_{\mathbb R}\frac{\mathbf 1(\ell_0<|\lambda| \leq \ell_1)}{\lambda}\,d\Xi_n(\lambda)\underset{n \rightarrow \infty}
{\longrightarrow} h_{\ell_0} - h_{\ell_1}.$$
Hence,
$$\underset{n \rightarrow \infty}{\lim} h_{n, \ell_0}
- \underset{n \rightarrow \infty}{\lim} h_{n, \ell_1} = h_{\ell_0} - h_{\ell_1}.$$
where, by assumption, the two limits in the left-hand side are well-defined. By
\eqref{e:hlim}, one deduces, by taking $\ell_1 \rightarrow \infty$,
$$\underset{n \rightarrow \infty}{\lim} h_{n, \ell_0} - h = h_{\ell_0},$$
which proves \eqref{convergenceh}.
In order to show \eqref{convergencegamma}, let us first check the following properties,
available for all $k \in \mathbb{Z}$:
\begin{itemize}
\item If $\lambda_k$ is well-defined, then $\lambda_{n,k}$ is well-defined for all
$n$ large enough and tends to $\lambda_{k}$ when $n$ goes to infinity.
\item If $\lambda_k$ is not well-defined, then for all $A > 0$, there are
finitely many indices $n$ such that $\lambda_{n,k}$ is well-defined and
in the interval $[-A,A]$.
\end{itemize}
By symmetry, we can assume that $k \geq 1$. We know that $\lambda^*$ is not
in $L$, and then for $\epsilon > 0$ small enough,
\begin{equation}
L \cap [\lambda^* - 3 \epsilon, \lambda^* + 3 \epsilon] = \emptyset, \label{nolambda*1}
\end{equation}
Let us fix $\epsilon > 0$ satisfying this property.
Since $\Xi_n$ tends locally weakly to $\Xi$, we deduce that
for $n$ large enough,
\begin{equation}
 L_n \cap [\lambda^* - 2 \epsilon, \lambda^* + 2\epsilon] = \emptyset, \label{nolambda*2}
 \end{equation}
which implies that $\lambda_k \geq \lambda_1 > \lambda^* + 2\epsilon$.
Now, let $\Phi$ and $\Psi$ be two continuous functions with compact support,
such that:
\begin{itemize}
\item For $\lambda \leq \lambda^* - \epsilon$, $$\Phi(\lambda) = \Psi(\lambda) = 0.$$
\item For $\lambda^* - \epsilon \leq  \lambda \leq \lambda^* + \epsilon$,
 $$0 \leq \Phi(\lambda) = \Psi(\lambda) \leq  1.$$
 \item For $\lambda^* + \epsilon \leq \lambda \leq \lambda_k - \epsilon$,
 $$\Phi(\lambda) = \Psi(\lambda) = 1$$ (recall that
  $\lambda^* + \epsilon <  \lambda_k - \epsilon$).
\item  For $\lambda_k - \epsilon \leq \lambda \leq \lambda_k$,
$$0 \leq \Phi(\lambda) \leq \Psi(\lambda) = 1.$$
\item For $\lambda_k  \leq \lambda \leq \lambda_k+ \epsilon$,
$$0 = \Phi(\lambda) \leq \Psi(\lambda) \leq 1.$$
\item For $ \lambda \geq \lambda_k+ \epsilon$,
$$\Phi(\lambda) = \Psi(\lambda) = 0.$$
\end{itemize}
By using \eqref{nolambda*1}, we deduce
$$\int_{\mathbb{R}} \Phi(\lambda) d \Xi(\lambda)
\leq \Xi ([\lambda^* - \epsilon, \lambda_k)) =
 \Xi ([\lambda^*, \lambda_k)) = k-1$$
 and
 $$\int_{\mathbb{R}} \Psi(\lambda) d \Xi(\lambda)
\geq \Xi ([\lambda^* + \epsilon, \lambda_k]) =
 \Xi ([\lambda^*, \lambda_k]) = k.$$
 Hence, for $n$ large enough,
 $$\int_{\mathbb{R}} \Phi(\lambda) d \Xi_n(\lambda) \leq k - 1/2$$
 and
 $$\int_{\mathbb{R}} \Psi(\lambda) d \Xi_n(\lambda) \geq k - 1/2,$$
 which implies
 $$ \Xi_n ([\lambda^*, \lambda_k - \epsilon)) =
  \Xi_n ([\lambda^*+ \epsilon, \lambda_k - \epsilon))
  \leq \int_{\mathbb{R}} \Phi(\lambda) d \Xi_n(\lambda) \leq k - 1/2$$
 and
 $$\Xi_n ([\lambda^*, \lambda_k + \epsilon)) =
  \Xi_n ([\lambda^*- \epsilon, \lambda_k + \epsilon])
  \geq \int_{\mathbb{R}} \Psi(\lambda) d \Xi_n(\lambda) \geq k - 1/2.$$
  Therefore, for $n$ large enough the point $\lambda_{n,k}$ is well-defined and between
  $\lambda_{k} - \epsilon$ and $\lambda_{k} + \epsilon$. Since $\epsilon$ and
  be taken arbitrarily small, we
  have proven the convergence claimed above in the case where $\lambda_k$ is
  well-defined. If $\lambda_k$ is not well-defined, let us choose $\epsilon > 0$
  satisfying \eqref{nolambda*1}, and $A > |\lambda^*|$. Let $\Phi$ be
  a continuous function with compact support, such that:
  \begin{itemize}
  \item For all $\lambda \in \mathbb{R}$, $\Phi(\lambda) \in [0,1]$.
  \item For all $\lambda \in [\lambda^*, A]$, $\Phi(\lambda) = 1$.
  \item For all $\lambda \notin (\lambda^*-\epsilon, A + \epsilon)$, $\Phi(\lambda) = 0$.
  \end{itemize}
Since  $\lambda_k$ is not well-defined,
$$\int_{\mathbb{R}} \Phi(\lambda) d \Xi(\lambda) \leq \Xi([\lambda^*- \epsilon, A+ \epsilon
]) = \Xi([\lambda^*, A +\epsilon]) \leq \Xi([\lambda^*, \infty)) \leq k-1,$$
and then for $n$ large enough,
$$\Xi_n([\lambda^*, A]) \leq  \int_{\mathbb{R}} \Phi(\lambda) d \Xi_n(\lambda) \leq k-1/2,$$
which implies that $\lambda_{n,k}$ cannot be well-defined and smaller than or equal to $A$.
This proves the second claim. Let us now go back to the proof of \eqref{convergencegamma}.
If $L \cap [-\ell_0, \ell_0] = \emptyset$, then $L \cap [-\ell_0- \epsilon, \ell_0+ \epsilon]
= \emptyset$ for some $\epsilon > 0$. Hence, there exists a nonnegative, continuous function
with compact support $\Phi$ such that $\Phi(\lambda) = 1$ for all $\lambda \in [-\ell_0, \ell_0]$,
 and
 $$\int_{\mathbb{R}} \Phi(\lambda) d \Xi(\lambda) = 0,$$
 which implies, for $n$ large enough,
 $$\Xi([-\ell_0, \ell_0]) \leq
 \int_{\mathbb{R}} \Phi(\lambda) d \Xi_n(\lambda) \leq 1/2,$$
i.e. $L_n \cap [-\ell_0, \ell_0] = \emptyset$. Hence, for $n$ large enough, the
two expressions involved in \eqref{convergencegamma} are identically zero.
If $L \cap [-\ell_0, \ell_0] \neq \emptyset$, let $k_1$ and $k_2$ be the smallest and
the largest indices $k$ such that $\lambda_k \in (-\ell_0, \ell_0)$.
Since $\lambda_{n,k_1}$ and $\lambda_{n,k_2}$ converge respectively to
$\lambda_{k_1}$ and $\lambda_{k_2}$ when $n$ goes to infinity, one has
$\lambda_{n,k_1}$ and $\lambda_{n,k_2}$ in the interval $(-\ell_0, \ell_0)$ for $n$ large enough.
On the other hand, $\lambda_{k_2+1}$ is either strictly larger than $\ell_0$ (strictly because by assumption,
$\ell_0 \notin L$), or not well-defined. In both cases, there are only finitely many indices $n$
such that $\lambda_{n,k_2+1} \leq \ell_0$. Similarly, by using the fact that
$-\ell_0 \notin L$, one checks that there are finitely many indices $n$ such that
$\lambda_{n,k_1-1} \geq -\ell_0$. Hence, for $n$ large enough, the indices $k$ such that
$\lambda_{n,k} \in [-\ell_0, \ell_0]$ are exactly the integers between $k_1$ and $k_2$,
which implies
$$ \int_{[-\ell_0, \ell_0]}
\frac{1}{\lambda-z}\,d\Lambda_n(\lambda) = \sum_{k= k_1}^{k_2}
\frac{\gamma_{n,k}}{\lambda_{n,k} - z},$$
whereas
$$ \int_{[-\ell_0, \ell_0]}
\frac{1}{\lambda-z}\,d\Lambda(\lambda) = \sum_{k= k_1}^{k_2}
\frac{\gamma_{k}}{\lambda_{k} - z}.$$
We have shown that for all $k$ between $k_1$ and $k_2$,
$\lambda_{n,k}$ tends to $\lambda_k$ when $n$ goes to infinity
and by assumption, $\gamma_{n,k}$ tends to $\gamma_k$. Moreover,
the numbers $\lambda_k$ are all distincts, and by assumption, $\gamma_k \neq 0$ for
all $k$. Hence, one can apply Lemma \ref{lemmarational} to deduce \eqref{convergencegamma}.

Let us now prove \eqref{convergenceT}. If $L \cap (\ell_0, \infty) = \emptyset$,
this statement can be deduced from the following convergences,
uniformly on $\{z \in \mathbb{C}, |z| < \ell_0/2\}$:
\begin{equation}
\mathds{1}_{L_n \cap (\ell_0, \infty) \neq \emptyset} \, \sum_{k \geq k_{n,\ell_0} } \frac{1}{\lambda_{n,k} (\lambda_{n,k} - z)}
\underset{n \rightarrow \infty}{\longrightarrow} 0, \label{convergenceterm1}
\end{equation}
\begin{equation}
\mathds{1}_{L_n \cap (\ell_0, \infty) \neq \emptyset} \, \frac{\Delta_{n,K_{n,\infty}}}{ \lambda_{n, K_{n,\infty}-1} - z }
\underset{n \rightarrow \infty}{\longrightarrow} 0, \label{convergenceterm2}
\end{equation}
\begin{equation}
\mathds{1}_{L_n \cap (\ell_0, \infty) \neq \emptyset} \, \frac{\Delta_{n,k_{n,\ell_0}}}{
 \lambda_{n, k_{n,\ell_0}} - z }
\underset{n \rightarrow \infty}{\longrightarrow} 0 \label{convergenceterm3}
\end{equation}
and
\begin{equation}
\mathds{1}_{L_n \cap (\ell_0, \infty) \neq \emptyset} \,
\sum_{ k \geq  k_{n,\ell_0}+1} \Delta_{n,k} \left(\frac{\lambda_{n,k}
 - \lambda_{n,k-1}}{(\lambda_{n,k-1} -z)
 (\lambda_{n,k} -z)} \right) \underset{n \rightarrow \infty}{\longrightarrow} 0. \label{convergenceterm4}
\end{equation}
If $L \cap (\ell_0, \infty) \neq \emptyset$, then we have proven previously that
$\lambda_{n,k_{\ell_0}}$ is well-defined for $n$ large enough and converges to
$\lambda_{k_{\ell_0}} > \ell_0$ when $n$ goes to infinity: in particular,
$\lambda_{n,k_{\ell_0}} > \ell_0$ for $n$ large enough. Moreover, one of the
two following cases occurs:
\begin{itemize}
\item If $\lambda_{k_{\ell_0} - 1}$ is well-defined, then it is strictly smaller
than $\ell_0$ (strictly because $\ell_0$ is, by assumption, not in $L$), and then
$\lambda_{n,k_{\ell_0} - 1}$ is, for $n$ large enough, well-defined and strictly
smaller than $\ell_0$.
\item If $\lambda_{k_{\ell_0} - 1}$ is not well-defined, and if $A > 0$,
 then for $n$ large enough,
$\lambda_{n,k_{\ell_0} - 1}$ is not well-defined or has an absolute value strictly
greater than $A$. By taking $A = \lambda_{k_{\ell_0}} + 1$, one deduces
that for $n$ large enough, $\lambda_{n,k_{\ell_0} - 1}$ is not-well defined,
strictly smaller than $- \lambda_{k_{\ell_0}} - 1$ or strictly
larger than $ \lambda_{k_{\ell_0}} + 1$. This last case is impossible for $n$
large enough, since $\lambda_{n,k_{\ell_0} - 1}$ is smaller than
 $\lambda_{n,k_{\ell_0}}$, which tends to $\lambda_{k_{\ell_0}}$. Hence, there are
 finitely many indices $n$ such that $\lambda_{n,k_{\ell_0} - 1}$ is well-defined
 and larger than $- \lambda_{k_{\ell_0}} - 1$, and a fortiori, larger than
 or equal to $\ell_0$.
\end{itemize}
All this discussion implies easily that for $n$ large enough, $k_{n,\ell_0} =
k_{\ell_0}$, and then it is sufficient to prove the uniform convergences
on $\{z \in \mathbb{C}, |z| < \ell_0/2\}$:
\begin{equation}
 \sum_{k \geq k_{\ell_0} } \frac{1}{\lambda_{n,k} (\lambda_{n,k} - z)}
\underset{n \rightarrow \infty}{\longrightarrow} \sum_{k \geq k_{\ell_0} } \frac{1}{\lambda_{k} (\lambda_{k} - z)}, \label{convergenceterm5}
\end{equation}
\begin{equation}
 \frac{\Delta_{n,k_{\ell_0}}}{ \lambda_{n, k_{\ell_0}} - z }
\underset{n \rightarrow \infty}{\longrightarrow}
\frac{\Delta_{k_{\ell_0}}}{ \lambda_{ k_{\ell_0}} - z }
 \label{convergenceterm6}
\end{equation}
and
\begin{align}
&  \frac{\Delta_{n,K_{n,\infty}}}{ \lambda_{n, K_{n,\infty}-1} - z } + \sum_{ k \geq  k_{\ell_0}+1} \Delta_{n,k} \left(\frac{\lambda_{n,k}
 - \lambda_{n,k-1}}{(\lambda_{n,k-1} -z)
 (\lambda_{n,k} -z)} \right) \nonumber \\ & \underset{n \rightarrow \infty}{\longrightarrow}
  \frac{\Delta_{K_{\infty}}}{
 \lambda_{ K_{\infty}-1} - z } +\sum_{ k \geq  k_{\ell_0}+1} \Delta_{k} \left(\frac{\lambda_{k}
 - \lambda_{k-1}}{(\lambda_{k-1} -z)
 (\lambda_{k} -z)} \right), \label{convergenceterm7}
\end{align}
with obvious notation.

Let us first prove \eqref{convergenceterm1}. If
$L \cap (\ell_0, \infty) = \emptyset$, then $L \cap (\ell_0 - \epsilon, \infty) = \emptyset$
for some $\epsilon > 0$ (recall that $\ell_0 \notin L$). Hence, for all $A > \ell_0$,
and $n$ large enough depending on $A$, $L_n \cap (\ell_0,A] = \emptyset$, which implies,
for $|z| \leq \ell_0/2$,
\begin{align*}
\left|\mathds{1}_{L_n \cap (\ell_0, \infty) \neq \emptyset} \, \sum_{k \geq k_{n,\ell_0} } \frac{1}{\lambda_{n,k} (\lambda_{n,k} - z)} \right|
& \leq 2 \int_{\mathbb{R}} \frac{\mathbf{1} (\lambda > \ell_0)}{\lambda^2} \, d \Xi_n(\lambda)
\\ & = 2 \int_{\mathbb{R}} \frac{\mathbf{1} (\lambda > A)}{\lambda^2} \, d \Xi_n(\lambda)
\\ & \leq 2\int_{\mathbb{R}} \frac{\mathbf{1} (\lambda > A)}{\lambda^{1+\alpha}} \, d \Xi_n(\lambda) \leq 2 \tau_A.
\end{align*}
By letting $n \rightarrow \infty$ and then $A \rightarrow \infty$, one deduces  \eqref{convergenceterm1}.

Let us prove \eqref{convergenceterm2} and \eqref{convergenceterm3}. By using
the estimates \eqref{tau} and \eqref{minorizationlambda} proven above, one deduces
that for
$$\tilde{\tau} := \tau_1 + \Xi([\lambda^* \wedge 2, 2]) + \sup_{n \geq 1}  \Xi_n([\lambda^* \wedge 2, 2]),$$
one has, for any $k \geq 1$,
\begin{equation}
\lambda_{k} \geq (k/\tilde{\tau})^{1/(1+\alpha)} \label{minorizationlambda2},
\end{equation}
if $\lambda_{k}> 2$, and uniformly in $n$,
\begin{equation}
\lambda_{k,n} \geq (k/\tilde{\tau})^{1/(1+\alpha)} \label{minorizationlambda3},
\end{equation}
if $\lambda_{k,n} > 2$. Now, let us assume that $L \cap (\ell_0, \infty) = \emptyset$ and
$L_n \cap (\ell_0, \infty) \neq \emptyset$. If $n$ is large enough, then for any index
$k$ such that $\lambda_{n,k} > \ell_0$, one has also $\lambda_{n,k} > \lambda^* \vee 2$,
since $L \cap (\ell_0-\epsilon, (\lambda^* \vee 2) + 1) =  \emptyset$ for some
$\epsilon > 0$, and $\Xi_{n} \rightarrow \Xi$. Hence,
$k \geq 1$ and \eqref{minorizationlambda3} is satisfied. By using
this inequality and \eqref{e:gammatight}, one deduces, for $|z| \leq \ell_0/2$,
\begin{align*}
& \left|\mathds{1}_{L_n \cap (\ell_0, \infty) \neq \emptyset} \, \frac{\Delta_{n,K_{n,\infty}}}{ \lambda_{n, K_{n,\infty}-1} - z } \right|   +
\left|\mathds{1}_{L_n \cap (\ell_0, \infty) \neq \emptyset} \, \frac{\Delta_{n,k_{n,\ell_0}}}{
 \lambda_{n, k_{n,\ell_0}} - z }\right| \\ &
  \leq \sup_{k \geq 1} \frac{2 c (k+1)^{\alpha'}} {(k/\tilde{\tau})^{1/(1+\alpha)} \vee  \lambda_{n, k_{n,\ell_0}} }
 +\sup_{k \geq 1} \frac{2 c k^{\alpha'}} {(k/\tilde{\tau})^{1/(1+\alpha)} \vee  \lambda_{n, k_{n,\ell_0}} }
\\ & \leq 4c (1+2^{\alpha'}) (1+\tilde{\tau})^{1/(1+\alpha)} \sup_{k \geq 1} \frac{k^{\alpha'}} {k^{1/(1+\alpha)} \vee  \lambda_{n, k_{n,\ell_0}} } \\ &  =  4c (1+2^{\alpha'}) (1+\tilde{\tau})^{1/(1+\alpha)}\sup_{k \geq 1} \frac{(k^{1/(1+ \alpha)})^{\alpha'(1+\alpha)}} {k^{1/(1+\alpha)} \vee  \lambda_{n, k_{n,\ell_0}} } \\ &  \leq  4c (1+2^{\alpha'}) (1+\tilde{\tau})^{1/(1+\alpha)}\sup_{k \geq 1}
 (k^{1/(1+\alpha)} \vee  \lambda_{n, k_{n,\ell_0}})^{\alpha'(1+\alpha) - 1}
 \\ & \leq  4c (1+2^{\alpha'}) (1+\tilde{\tau})^{1/(1+\alpha)}\lambda_{n, k_{n,\ell_0}}^{\alpha'(1+\alpha) - 1},
 \end{align*}
 where $\lambda_{n, k_{n,\ell_0}}$ is taken equal to $\infty$ for
 $L_n \cap (\ell_0, \infty) = \emptyset$. Note that in the previous computation, the last inequality is a consequence of the inequality $\alpha'(1+\alpha) - 1 <0$. Now, $\lambda_{n, k_{n,\ell_0}}$ tends to infinity with
 $n$, since for all $A > \ell_0$, one has $L_n \cap (\ell_0, A] = \emptyset$ for $n$ large enough. Hence,
 we get \eqref{convergenceterm2} and \eqref{convergenceterm3}.

 Moreover, in case where $L \cap (\ell_0, \infty) = \emptyset$,
  $L_n \cap (\ell_0, \infty) \neq \emptyset$, $n$ is large enough, and $|z| < \ell_0/2$, the left-hand side of \eqref{convergenceterm4}
 is smaller than or equal to:
 \begin{align*}
4 c & \sum_{ k \geq  k_{n,\ell_0}+1} |k|^{\alpha'} \left(\frac{\lambda_{n,k}
 - \lambda_{n,k-1}}{\lambda_{n,k-1}
 \lambda_{n,k} } \right)  =
 4 c  \sum_{ k \geq  k_{n,\ell_0}+1} |k|^{\alpha'} \left(\frac{1}{\lambda_{n,k-1}}
 -\frac{1}{\lambda_{n,k}} \right)
 \\ & = 4c \left(\frac{|k_{n, \ell_0}+1|^{\alpha'}}{\lambda_{n,k_{n, \ell_0}}}
 +  \sum_{ k \geq  k_{n,\ell_0}+1} \frac{|k+1|^{\alpha'} - |k|^{\alpha'}}{\lambda_{n,k}} \right)
\end{align*}
$$ \leq 4c \left(\frac{|k_{n, \ell_0}+1|^{\alpha'}}{\lambda_{n,k_{n, \ell_0}} \vee
 (|k_{n, \ell_0}|/\tilde{\tau})^{1/(1+\alpha)} }
 +  \sum_{ k \geq  1} \frac{(k+1)^{\alpha'} - k^{\alpha'}}{\lambda_{n,k_{n, \ell_0}}
 \vee (k/\tilde{\tau})^{1/(1+\alpha)}} \right),$$
 when $k_{n, \ell_0} \geq 1$, which occurs for $n$ large enough.  The first term of the last quantity is dominated by
 $$(\lambda_{n,k_{n, \ell_0}} \vee (k_{n, \ell_0}/\tilde{\tau})^{1/(1+\alpha)})^{\alpha' (1+\alpha)-1}
 \leq  (\lambda_{n,k_{n, \ell_0}})^{\alpha'(1+\alpha)-1},$$
 which tends to zero when $n$ goes to infinity, since $\lambda_{n,k_{n, \ell_0}}$ goes to infinity
 and $\alpha'(1+\alpha)-1 < 0$. Similarly,
 $$ \sum_{ k \geq  1} \frac{(k+1)^{\alpha'} - k^{\alpha'}}{\lambda_{n,k_{n, \ell_0}}
 \vee (k/\tilde{\tau})^{1/(1+\alpha)}} \underset{n \rightarrow \infty}{\longrightarrow} 0,$$
 by dominated convergence. Hence, we get \eqref{convergenceterm4}.

We can now assume $L \cap (\ell_0, \infty) \neq \emptyset$ and it remains to prove \eqref{convergenceterm5},
\eqref{convergenceterm6} and \eqref{convergenceterm7}.

For $k \geq k_{\ell_0}$, let us define $\lambda_{n,k}$ and $\lambda_{k}$ as $\infty$ if these
numbers are not well-defined: this does not change the quantities involved in \eqref{convergenceterm5}.
Moreover, for all $k \geq k_{\ell_0}$:
\begin{itemize}
\item If $\lambda_{k}$ is well-defined as a finite quantity,
then $\lambda_{n,k}$ is also well-defined for $n$ large enough and tends to $\lambda_{k}$ when
$n$ goes to infinity.
\item If $\lambda_k = \infty$, then for all $A >0$,
and for $n$ large enough, one has $\lambda_{n,k} \notin [-A,A]$. Since for $n$ large enough,
$$\lambda_{n,k} \geq \lambda_{n, k_{\ell_0}}   > \lambda_{k_{\ell_0}} -1 > \ell_0 - 1 > 0,$$
one has  $\lambda_{n,k} > A$: in other words, $\lambda_{n,k}$ tends to infinity with $n$.
 \end{itemize}
 We have just checked that with the convention made here, one has always
 $\lambda_{n,k}$ converging to $\lambda_{k}$ when $n$ goes to infinity, for all $k \geq k_{\ell_0}$.
 Hence, \eqref{convergenceterm5} is a consequence of the dominated convergence theorem and
 the majorization:
 $$\sum_{k \geq k_{\ell_0}} (\lambda_k \wedge \inf_{n \geq n_0} \lambda_{n,k})^{-2} < \infty,$$
 for some $n_0 \geq 1$.
 Now, there exists $n_0 \geq 1$ such that for all $n \geq n_0$, one has $k_{n,\ell_0} = k_{\ell_0}$, and then for
 all $k \geq 1 \vee k_{\ell_0}$, $\lambda_{k} > \ell_0 > 2$, $\lambda_{n,k} >2$ and $k \geq 1$, which
  implies the minorizations \eqref{minorizationlambda2} and \eqref{minorizationlambda3}.
  Hence one gets \eqref{convergenceterm5}, since
   $$\sum_{k \geq 1} (k/\tilde{\tau})^{-2/(1+\alpha)} < \infty.$$
   Since \eqref{convergenceterm6} is easy to check, it remains to show \eqref{convergenceterm7}, which can
   be rewritten as follows:
 $$  \sum_{ k \geq  k_{\ell_0}+1} \Delta_{n,k} \left(\frac{1}{\lambda_{n,k-1} -z}
 - \frac{1}{\lambda_{n,k} - z} \right)
 \underset{n \rightarrow \infty}{\longrightarrow}
 \sum_{ k \geq  k_{\ell_0}+1} \Delta_{k} \left(\frac{1}{\lambda_{k-1} -z}
 - \frac{1}{\lambda_{k} - z} \right),
  $$
  where for $k \geq K_{n,\infty}$ (resp. $k \geq K_{\infty}$), one defines $\lambda_{k,n} := \infty$
  (resp. $\lambda_{k} := \infty$). Note that with this convention,
  $\lambda_{n,k}$ tends to $\lambda_k$ when $n$ goes to infinity, for all $k \geq k_{\ell_0}$.
  Note that each term of the left-hand side of this last convergence converges uniformly
  on $\{z \in \mathbb{C}, |z| < \ell_0/2\}$ towards the corresponding term in the
  right-hand side. Indeed, for $n$ large enough, for all $k \geq k_{\ell_0}$, for $|z| < \ell_0/2$,
  and for $\lambda_k$, $\lambda_{n,k}$ finite,
 \begin{align*}
 \left|\frac{1}{\lambda_{n,k} - z} - \frac{1}{\lambda_{k} - z} \right|
 & = \frac{|\lambda_{k} - \lambda_{n,k}|}{|\lambda_{n,k} - z||\lambda_{k} - z|}
 \leq \frac{4 |\lambda_{k} - \lambda_{n,k}|}{\lambda_{n,k} \lambda_{k} }
 \\ & \leq 4 \left| \frac{1}{\lambda_{n,k} } - \frac{1}{\lambda_{k}}\right|
 \underset{n \rightarrow \infty}{\longrightarrow} 0,
 \end{align*}
 this convergence, uniform in $z$, being in fact also true if $\lambda_{n,k}$ or $\lambda_k$ is infinite.
 Hence, one has, for all $k' > k_{\ell_0}+1$, the uniform convergence:
 $$  \sum_{ k_{\ell_0}+1 \leq k \leq k'} \Delta_{n,k} \left(\frac{1}{\lambda_{n,k-1} -z}
 - \frac{1}{\lambda_{n,k} - z} \right)
 \underset{n \rightarrow \infty}{\longrightarrow}
 \sum_{k_{\ell_0}+1 \leq k \leq k'} \Delta_{k} \left(\frac{1}{\lambda_{k-1} -z}
 - \frac{1}{\lambda_{k} - z} \right),
  $$
   Hence, it is sufficient to check, for $n_0 \geq 1$ such that $k_{n,\ell_0} =k_{\ell_0}$ if
 $n \geq n_0$, that
\begin{equation}
 \sup_{n \geq n_0, |z| < \ell_0/2}
 \left| \sum_{ k  > k'} \Delta_{n,k} \left(\frac{1}{\lambda_{n,k-1} -z}
 - \frac{1}{\lambda_{n,k} - z} \right)  \right| \underset{k' \rightarrow \infty}{\longrightarrow} 0
 \label{convergencek'1}
\end{equation}
and
\begin{equation}
 \sup_{|z| < \ell_0/2}
 \left| \sum_{ k  > k'} \Delta_{k} \left(\frac{1}{\lambda_{k-1} -z}
 - \frac{1}{\lambda_{k} - z} \right) \right| \underset{k' \rightarrow \infty}{\longrightarrow} 0.
 \label{convergencek'2}
\end{equation}
Now, for $k' \geq 1 \vee (k_{\ell_0} + 1)$, $n \geq n_0$ and $|z| < \ell_0/2$, one has
\begin{align*}
 \left| \sum_{ k  > k'} \Delta_{n,k} \left(\frac{1}{\lambda_{n,k-1} -z}
 - \frac{1}{\lambda_{n,k} - z} \right)  \right| &
 \leq \sum_{ k  > k'} |\Delta_{n,k}| \, \left|\frac{1}{\lambda_{n,k-1} -z}
 - \frac{1}{\lambda_{n,k} - z}  \right|
 \\ & \leq 4 c \sum_{ k  > k'} k^{\alpha'}  \, \left( \frac{1}{\lambda_{n,k-1}}
 - \frac{1}{\lambda_{n,k}} \right)
 \\ & = 4 c \left( \frac{(k'+1)^{\alpha'}}{\lambda_{n,k'}} +
 \sum_{k > k'}  \frac{(k+1)^{\alpha'} - k^{\alpha'}}{\lambda_{n,k}}  \right)
 \\ &  \leq 4 c \left( \frac{(k'+1)^{\alpha'}}{(k'/\tilde{\tau})^{1/(1+\alpha)}} +
 \sum_{k > k'}  \frac{(k+1)^{\alpha'} - k^{\alpha'}}{(k/\tilde{\tau})^{1/(1+\alpha)}}   \right)
 \\ & \underset{k' \rightarrow \infty}{\longrightarrow} 0,
\end{align*}
  which proves \eqref{convergencek'1}. One shows \eqref{convergencek'2} in an exactly similar way,
  which finishes the proof of the convergence of $S_n$ towards $S+\bar{\gamma}h$, uniformly on compact sets for the
  distance $d$. It remains to prove that $\Xi'_{n}$ and $\Xi'$ satisfy the assumptions from \eqref{e:lambda1} to
\eqref{e:hlim}. Note that the observation of the sign of the imaginary parts $\Im(S_n)$ and $\Im(S)$ implies
 that the sets $\Xi'_n$ and $\Xi'$ are included in $\mathbb{R}$. Moreover, the derivatives $S_n'$ and $S'$
 are strictly positive, respectively on $\mathbb{R} \backslash L_n$ and $\mathbb{R} \backslash L$, and
 all the left (resp. right) limits of $S_n$ and $S$ at their poles are equal to $+ \infty$ (resp. $-\infty$).
 We deduce that the support of $\Xi'_n$ (resp. $\Xi'$) strictly interlaces with the points
 in $L_n$ (resp. $L$).

 The convergence \eqref{e:lambda1} is a direct consequence of the convergence of
$S_n$ towards $S+\bar{\gamma}h$, as written in the statement of Theorem \eqref{topologytheorem}. More
precisely, for two points $a$ and $b$ ($a < b$) not in the support of $\Xi'$
and such that $\Xi'((a,b)) = k$, there exist real numbers $a = q_0 < q_1 < q_2 < \dots < q_{2k}< b = q_{2k+1}$ such that $-\infty < S(q_{2j-1}) < h' < S(q_{2j}) < \infty$ for $ j \in \{1, \dots, k\}$, which implies that
these inequalities are also satisfied for $S_n - \bar{\gamma}h$ instead of $S$ if $n$ is large enough:
one has $\Xi'_n((a,b)) \geq k$. On the other hand, since the support of $\Xi'$
has exactly one point on each interval $[q_{2j-1}, q_{2j}]$ ($1 \leq j \leq k$) and no point on
the intervals $[q_{2j},q_{2j+1}]$ ($0 \leq j \leq k$), one deduces that $S$ is bounded
on the intervals $[q_{2j-1}, q_{2j}]$ and bounded away from $h'$ on the intervals $[q_{2j},q_{2j+1}]$.
These properties remain true for $S_n - \bar{\gamma}h$ if $n$ is large enough, and one
easily deduces that $\Xi'_n((a,b)) \leq k$.

The properties \eqref{e:lambdasquared}, \eqref{e:lambdasquared2} and \eqref{e:hlim} can be deduced from the property of
interlacing.
More precisely, for $\ell \geq 1$,
$$\int_{\mathbb{R}} \frac{\mathbf{1}(\lambda > \ell)}{\lambda^{1+\alpha}} \, d \Xi'_n(\lambda)
= \sum_{\lambda \in L'_n \cap (\ell, \infty)} \frac{1}{\lambda^{1+\alpha}},$$
where $L'_n$ is the support of $\Xi'_n$. By the interlacing property,
if $L'_n \cap (\ell, \infty)$ is not empty and if its smallest element is
$\lambda' > \ell$, then it is possible to define an injection
 between  $(L'_n \cap (\ell, \infty) ) \backslash \{\lambda'\}$ and
 $L_n \cap (\ell, \infty)$, such that the image of each point is smaller than this point.
 One deduces
 $$\int_{\mathbb{R}} \frac{\mathbf{1}(\lambda > \ell)}{\lambda^{1+\alpha}} \, d \Xi'_n(\lambda)
 \leq \frac{1}{\lambda'} + \sum_{\lambda \in L_n \cap (\ell, \infty)} \frac{1}{\lambda^{1+\alpha}}
 \leq \frac{1}{\ell} + \int_{\mathbb{R}} \frac{\mathbf{1}(\lambda > \ell)}{\lambda^{1+\alpha}} \, d \Xi_n(\lambda).$$
 By looking similarly at the integral for $\lambda  < -\ell$, one deduces
 $$\int_{\mathbb{R}} \frac{\mathbf{1}(|\lambda| > \ell)}{|\lambda|^{1+\alpha}} \, d \Xi'_n(\lambda)
  \leq \frac{2}{\ell} + \int_{\mathbb{R}} \frac{\mathbf{1}(|\lambda| > \ell)}{|\lambda|^{1+\alpha}} \, d \Xi_n(\lambda) \leq \frac{2}{\ell} +\tau_{\ell} \underset{\ell \rightarrow \infty}{\longrightarrow} 0,$$
  which proves \eqref{e:lambdasquared} for the measure $\Xi'_n$.

  By a similar argument, for $\ell'' > \ell' > \ell \geq 1$,
 $$\int_{\mathbb{R}}  \frac{\mathbf{1}(\ell' \leq \lambda< \ell'')}{\lambda} \, d \Xi'_n(\lambda)
 \leq \int_{\mathbb{R}}  \frac{\mathbf{1}(\ell' \leq  \lambda< \ell'')}{\lambda} \, d \Xi_n(\lambda)
 + \frac{1}{\ell'},$$
 and one has the similar inequalities obtained by exchanging $\Xi_n$ and $\Xi'_n$, and
 by changing the sign of $\lambda$.
 Hence,
 $$ \left|\int_{\mathbb{R}}  \frac{\mathbf{1}(\ell' \leq |\lambda|< \ell'')}{|\lambda|} \, d \Xi'_n(\lambda)
 - \int_{\mathbb{R}}  \frac{\mathbf{1}(\ell' \leq |\lambda|< \ell'')}{|\lambda|} \, d \Xi_n(\lambda)
 \right| \leq \frac{4}{\ell'}.
$$
 This inequality and the existence of the limit $h_{n,\ell}$ for the measure
 $\Xi_n$ implies that
 $$\underset{\ell' \wedge \ell'' \rightarrow \infty}{\limsup}
 \left|\int_{\mathbb{R}}  \frac{\mathbf{1}(\ell <  |\lambda|< \ell'')}{|\lambda|} \, d \Xi'_n(\lambda)
 - \int_{\mathbb{R}}  \frac{\mathbf{1}(\ell <  |\lambda|< \ell')}{|\lambda|} \, d \Xi'_n(\lambda) \right|
 = 0,$$
and then the limit given by \eqref{e:lambdasquared2} exists for the measure $\Xi'_n$.
Moreover, for all $\ell \geq 1$, one gets the majorization:
\begin{equation}\left|\underset{\ell' \rightarrow \infty}{\lim}
 \int_{\mathbb{R}}  \frac{\mathbf{1}(\ell <  |\lambda|< \ell')}{|\lambda|} \, d \Xi'_n(\lambda)
 - \underset{\ell' \rightarrow \infty}{\lim}
 \int_{\mathbb{R}}  \frac{\mathbf{1}(\ell <  |\lambda|< \ell')}{|\lambda|} \, d \Xi_n(\lambda)
 \right| \leq \frac{4}{\ell},
 \label{xxxx2017}
 \end{equation}
 and a similar inequality without the index $n$. This implies \eqref{e:hlim}, provided that we check the
 existence of the limit:
\begin{equation}
 \lim_{n \rightarrow \infty}  \underset{\ell' \rightarrow \infty}{\lim}
 \int_{\mathbb{R}}  \frac{\mathbf{1}(\ell <  |\lambda|< \ell')}{|\lambda|} \, d \Xi'_n(\lambda) \label{existencelimit}
 \end{equation}
 for each $\ell$ such that $\ell$ and $-\ell$ are not in the support of  $\Xi'$. Let us first assume that $\ell$ and
 $-\ell$ are also not in the
 support of $\Xi$.
 We have, for all $\ell'' > \ell$,
\begin{equation}
\label{xxxx2018}\underset{\ell' \rightarrow \infty}{\lim}
 \int_{\mathbb{R}}  \frac{\mathbf{1}(\ell <  |\lambda|< \ell')}{|\lambda|} \, d \Xi'_n(\lambda)
 = \int_{\mathbb{R}}  \frac{\mathbf{1}(\ell <  |\lambda| \leq \ell'')}{|\lambda|} \, d \Xi'_n(\lambda)
 +  \underset{\ell' \rightarrow \infty}{\lim}
 \int_{\mathbb{R}}  \frac{\mathbf{1}(\ell''<  |\lambda|< \ell')}{|\lambda|} \, d \Xi'_n(\lambda),
 \end{equation}
 and a similar equality with $\Xi'_n$ replaced by $\Xi_n$.
 Since $\ell$ and $-\ell$ are not in the support of $\Xi$ or $\Xi'$, the convergences of
$\Xi_n$ towards $\Xi$ and of $\Xi'_n$ towards $\Xi'$ imply that the lower and upper limits (when $n$ goes to infinity)
of the first term of \eqref{xxxx2018} (both with $\Xi'_n$ and with $\Xi_n$) differ by $O(1/\ell'')$.
For the second term, the difference between the lower and upper limits should change only by
$O(1/\ell'')$ when we replace \eqref{xxxx2018} by the same equation with $\Xi_n$, thanks to
 \eqref{xxxx2017}. Hence, this observation is also true for the sum of the two terms.
 On the other hand, the existence of the limit of $h_{n,\ell}$ when $n$ goes to infinity (for $\Xi_n$) implies that
 in \eqref{xxxx2018} with $\Xi'_n$ replaced by $\Xi_n$, the difference between the upper and lower limit is zero.
 Therefore, the difference is
 $O(1/\ell'')$ without replacement of $\Xi'_n$ by $\Xi_n$ : letting $\ell'' \rightarrow 0$ gives the existence of the  limit
 \eqref{existencelimit} for $-\ell, \ell$ not in the support of $\Xi$ and $\Xi'$. If $-\ell$ or $\ell$
 is in the support of $\Xi$ (but not in the support of $\Xi'$), we observe that for some $\epsilon > 0$ and
  $n$ large enough,
 there is no point in the supports of $\Xi'_n$ and $\Xi'$ in the intervals $\pm \ell + (-\epsilon, \epsilon)$, which implies
 that the integral involved in \eqref{existencelimit} does not change if we change $\ell$ by less than $\epsilon$.  By suitably moving $\ell$, we can then also avoid the support of $\Xi$.

\end{proof}

 \section{Convergence of Hermite corners towards the bead process}
 
 In this section, we consider, for all $\beta > 0$, the Gaussian $\beta$ Ensemble, defined 
 as a a set of $n$ points 
 $(\lambda_j)_{1 \leq j \leq n}$ whose  joint density, with respect to the Lebesgue measure
is  proportional to $$e^{-\beta \sum_{k=1}^n \lambda_k/4}
\prod_{j < k} |\lambda_j - \lambda_k|^{\beta}.$$

 We will use the following crucial estimate, proven in \cite{NV19}:
\begin{theorem} \label{boundvariance2017}
For $-\infty \leq \Lambda_1 < \Lambda_2 \leq \infty$, let $N(\Lambda_1, \Lambda_2)$
be the number of points, between $\Lambda_1$ and $\Lambda_2$, of a Gaussian beta ensemble with $n$
points, and let $N_{sc}(\Lambda_1, \Lambda_2)$
be $n$ times the measure of $(\Lambda_1, \Lambda_2)$
with respect to the semi-circle distribution on the interval
$[- 2\sqrt{n}, 2 \sqrt{n}]$:
$$N_{sc}(\Lambda_1, \Lambda_2) :=
\frac{n}{2\pi} \int_{\Lambda_1/\sqrt{n}}^{\Lambda_2/\sqrt{n}}
\sqrt{(4 - x^2)_+}  \,  dx.
$$
Then,
$$\mathbb{E} [(N(\Lambda_1, \Lambda_2) - N_{sc}(\Lambda_1, \Lambda_2))^2] = O( \log (2 + (\sqrt{n}(\Lambda_2 - \Lambda_1) \wedge n))).$$
\end{theorem}
For $\beta \in \{1,2,4\}$, the Gaussian $\beta$ Ensemble can be represented by the
  eigenvalues of  real symmetric (for $\beta = 1$), complex  Hermitian  (for $\beta = 2$), or quaternionic Hermitian (for $\beta = 4$) Gaussian matrices. The law of the entries
of these matrices, corresponding respectively to the Gaussian Orthogonal Ensemble, the Gaussian Unitary Ensemble and the Gaussian Symplectic Ensemble,  are given as follows:
\begin{itemize}
\item The diagonal entries are real-valued, centered, Gaussian with variance $2/\beta$.
\item The entries above the diagonal are real-valued for $\beta = 1$, complex-valued for $\beta = 2$, quaternion-valued for $\beta = 4$, with independent parts, centered, Gaussian with variance $1/\beta$.
\item All the entries involved in the previous items are independent.
\end{itemize}
By considering the top-left minors $A_n$ of an infinite random matrix $A$ following the law described just above, and their eigenvalues, we get a family of sets of points, the $n$-th set following the G$\beta$E of order $n$.
Conditionally on the matrix $A_n$, whose eigenvalues are denoted $(\lambda_1, \dots, \lambda_n)$, supposed to be distinct (this holds almost surely), the law of the eigenvalues of $A_{n+1}$ can be deduced by diagonalizing $A_n$ inside $A_{n+1}$, which gives a matrix of  the form
$$ \left( \begin{array}{ccccc}
\lambda_1 & 0  & \cdots & 0 &  g_1 \\
0 & \lambda_2 & \cdots & 0 &  g_2 \\
\vdots & \vdots & \ddots & \vdots & \vdots
\\ 0 & 0 & \cdots & \lambda_n & g_n \\
 \overline{g_1} & \overline{g_2} & \cdots
& \overline{g_n} & g \end{array} \right),$$
 where $g_1, \dots g_n, g$ are independent, centered Gaussian, $g$ being real-valued with variance $2/\beta$, $g_1, \dots, g_n$ being real-valued of variance $1$ for $\beta = 1$, complex-valued with independent real and imaginary parts of variance $1/2$ for $\beta = 2$, quaternion-valued with independent parts of variance $1/4$ for $\beta = 4$. 
 Expanding the characteristic polynomial and dividing by the product of $\lambda_j - z$ for $1 \leq j \leq \lambda_n$, we see that the eigenvalues of $A_{n+1}$ are the solutions of the equation:
 $$ g - z - \sum_{j=1}^n \frac{|g_j|^2}{ \lambda_j - z}  = 0.$$
 Hence, if for $n \geq 1$, we consider the eigenvalues of the matrices $(A_{n+k})_{k \geq 0}$, we
  get an inhomogeneous Markov chain defined as follows:
 \begin{itemize}
 \item The first set  corresponds to the $G \beta E$ with $n$ points.
 \item Conditionally on the sets of points indexed by $0, 1, \dots, k$, the set indexed by $k$ containing the distinct points $\lambda_1, \dots, \lambda_{n+k}$, the  set indexed by $k+1$ contains the zeros of
 $$g - z - \sum_{j=1}^{n+k} \frac{(2/\beta) \gamma_{j}}{ \lambda_j - z},$$
 $g$ being centered, Gaussian of variance $2/\beta$, $\gamma_j$ being a Gamma variable of parameter $\beta/2$, all these variables being independent.
 \end{itemize}
 This Markov chain can be generalized to all $\beta > 0$: this can be viewed as the "eigenvalues of the G$\beta$E minors".
 In fact, what we obtain is equivalent (with suitable scaling) to the Hermite $\beta$ corners introduced by Gorin and Shkolnikov in \cite{GS}. 
 This fact is due to the following result, proven (up to scaling) in \cite{Forrester}, Proposition 4.3.2: 
 
 \begin{proposition}
 The density of transition probability from the set $(\lambda_1, \dots, \lambda_n)$ to the set $(\mu_1, \dots, \mu_{n+1})$, subject to the interlacement property
$$\mu_1 < \lambda_1 < \mu_2 < \dots < \mu_{n} < \lambda_n < \mu_{n+1},$$
is proportional to 
 $$\prod_{1 \leq p < q \leq n+1} (\mu_q - \mu_p) \prod_{1 \leq p < q \leq n} (\lambda_q - \lambda_p)^{1- \beta}
 \prod_{1 \leq p \leq n, 1 \leq q \leq n+1} |\mu_q - \lambda_p|^{\beta/2 - 1}e^{- \frac{\beta}{4} \left( \sum_{1 \leq q \leq n+1} \mu_q^2 -  \sum_{1 \leq p \leq n} \lambda_p^2 \right)}.$$
 \end{proposition}

As explained in \cite{GS}, the marginals of the Hermite $\beta$ corner correspond to the Gaussian $\beta$ Ensemble, which implies the following:
\begin{proposition}
 For all $\beta > 0$, the set of $n+k$ points corresponding to the step $k$ of the Markov chain just above has the distribution of the Gaussian $\beta$ Ensemble of dimension $n+k$.
In particular, if we take $n=1$, we get a coupling of the $G \beta E$ in all dimensions.
  \end{proposition}

Now, we show that a suitable scaling limit of this Markov chain is the $\beta$-bead process introduced in the paper.

We choose $\alpha \in (-2,2)$ (this corresponds to the bulk of the spectrum), $n \geq 1$, and we center the spectrum around the level $\alpha \sqrt{n}$. The expected density of eigenvalues around this level is approximated by
$\sqrt{n} \rho_{sc}(\alpha)$, where $\rho_{sc}$ is the density of the semi-circular distribution. In order to get an average spacing of $2\pi$, we should then scale the eigenvalues by a factor $2 \pi \sqrt{n} \rho_{sc}(\alpha) = \sqrt{n(4 - \alpha^2)}$.
For $k \geq 0$, we then consider the simple point measure $\Xi_n^{(k)}$ given by
putting Dirac masses at the points
$( \lambda^{(n,k)}_j - \alpha \sqrt{n}) \sqrt{n (4 - \alpha^2)}$, where $(\lambda^{(n,k)}_j)_{1 \leq j \leq n+k}$ is the set of $n+k$ points obtained at the  step $k$ of  the Markov chain above.
The sequence of measures $\Xi_n^{(k)}$  can be recovered as follows:
\begin{itemize}
\item For $k = 0$, $\Xi_n^{(0)}$ corresponds to the point measure associated with the suitably rescaled G$\beta$E point process, with $n$ points.
\item Conditionally on $\Xi_n^{(k)}$,
$\Xi_n^{(k+1)}$ is obtained by taking the zeros of
$$ - \frac{\alpha}{ \sqrt{4 - \alpha^2}}
+ \frac{g^{(k)}}{\sqrt{n(4 - \alpha^2)}}
- \frac{z}{n(4-\alpha^2)}
- \int \frac{1}{ \lambda- z} d \Lambda_n^{(k)} (\lambda),$$
where $g^{(k)}$ is a centered Gaussian variable of variance $2/\beta$, and $\Lambda_n^{(k)}$ is the weighted version of $\Xi_n^{(k)}$, the weights being i.i.d. with distribution corresponding to $2/\beta$ times a Gamma variable of parameter $\beta/2$.
\end{itemize}
We are now able to prove the following result:
\begin{theorem}
The Markov chain $(\Xi_n^{(k)})_{k \geq 0}$ converges in law to the Markov chain defined in Theorem \ref{betainfinite},
for the topology of locally weak convergence of locally finite measures on $\mathbb{R} \times \mathbb{N}_0$, and for
the level $$ h = - \frac{\alpha}{ \sqrt{4 - \alpha^2}}.$$
For $\beta = 2$ and $h \in \mathbb{R}$ fixed, the law of the Markov chain of Theorem \ref{betainfinite} corresponds (after dividing the points by $2$) to the bead process introduced
by Boutillier, with parameter
$$\gamma = - \frac{h}{\sqrt{1  + h^2}},$$
if we take the notation of \cite{Bou}.
\end{theorem}
\begin{proof}
By the result of Valk\'o and Vir\'ag, $\Xi_n^{(0)}$ converges in distribution to the $\operatorname{Sine}_{\beta}$ point process.

Hence, the family, indexed by $n$, of the distributions of $(\Xi_n^{(0)})_{n \geq 1}$, is tight in the space of probability measures on
$\mathcal{M}(\mathbb{R})$, $\mathcal{M}(\mathbb{R})$ being the space of locally finite measures on the Borel sets of $\mathbb{R}$, endowed with the topology of locally weak convergence.
Hence, for $\epsilon > 0$, there exists
$(C_K)_{K \in \mathbb{N}}$ such that with probability at least $1 - \epsilon$,
the number of points in $[-K,K]$
of $\Xi_n^{(0)}$ is at most $C_K$ for all $K \in \mathbb{N}$, independently of $n$.
Since the points of $\Xi^{(k)}_n$ interlace with those of
  $\Xi^{(k-1)}_n$, the condition just above is satisfied  with $\Xi^{(k)}_n$ instead of $\Xi^{(0)}_n$.
  Hence, the family, indexed by $n$, of the laws of $(\Xi^{(k)}_n)_{k \geq 0}$ is tight in the space of probability measures
  on $\mathcal{M}(\mathbb{R} \times \mathbb{N}_0)$,  $\mathcal{M}(\mathbb{R} \times \mathbb{N}_0)$ being the space of locally finite measures on $\mathbb{R} \times \mathbb{N}_0$, again endowed with the topology of locally weak convergence.
   From the tightness, it is enough to prove that the law of the Markov chain of Theorem  \ref{betainfinite} is the only possible limit for a subsequence of the laws of $(\Xi^{(k)}_n)_{k \geq 1}$.
Let us consider such a subsequence which converges in law.
We define the following random variable
$$Y_n := \sup_{\ell \geq 0}
( 1+ \ell)^{-3/4} (|\widetilde{\Xi}^{(0)}_n ([0, \ell])  |
+ |\widetilde{\Xi}^{(0)}_n ([-\ell, 0])  |)$$
where
$$\widetilde{\Xi}^{(0)}_n ([a,b]) :=
\Xi^{(0)}_n([a,b]) - N_{sc}
\left(\left[\alpha \sqrt{n} + \frac{a}{\sqrt{n(4-\alpha^2)}},\alpha \sqrt{n} + \frac{b}{\sqrt{n(4-\alpha^2)}} \right] \right),$$
for
$$N_{sc} (\Lambda_1, \Lambda_2)
:= \frac{n}{2 \pi} \int_{\Lambda_1/\sqrt{n}}^{\Lambda_2/\sqrt{n}}
\sqrt{(4-x^2)_+} dx.$$
The family $(Y_n)_{n \geq 0}$ is tight.
Indeed,
by Theorem \ref{boundvariance2017},
$$\mathbb{E} [(1+ \ell)^{-3/2}
(|\widetilde{\Xi}^{(0)}_n ([0, \ell])  |^2
+ |\widetilde{\Xi}^{(0)}_n ([-\ell, 0])|^2 )] =
O \left( (1+ \ell)^{-3/2}
\log \left( 2 + \sqrt{n} \frac{\ell}{\sqrt{n(4-\alpha^2)}}  \right) \right),$$
which shows that
$$\mathbb{E}
\left[ \sum_{\ell=0}^{\infty}
(1+ \ell)^{-3/2}
(|\widetilde{\Xi}^{(0)}_n ([0, \ell])  |
+ |\widetilde{\Xi}^{(0)}_n ([-\ell, 0])|)^2  \right] \leq C_{\alpha,\beta},$$
where $C_{\alpha,\beta} < \infty$ depends only on $\alpha$ and $\beta$ (in particular, not on $n$). This implies that $(Y_n)_{n \geq 1}$ is tight.

The point processes $\Xi_n^{(k)}$, $k \geq 0$, are constructed from $\Xi_n^{(0)}$, and families $\gamma_{n,k,k'}$ of weights, $\gamma_{n,k,k'}$ being
 the weight, involved in the construction  of $\Xi_n^{(k+1)}$,
  of the $(k')$-th nonnegative point
 of $\Xi_n^{(k)}$ if $k' > 0$, the $(1-k')$-th negative point of $\Xi_n^{(k)}$ if $k' \leq 0$.
 All the variables $\gamma_{n,k,k'}$  are i.i.d., distributed like $2/\beta$ times a Gamma variable of parameter $\beta/2$, and independent of $\Xi_n^{(0)}$.
 We can consider the variables
 $$Z_{n,k} = \sup_{m \geq 1} m^{-0.51}
 \left| m - \sum_{k' = 0}^{m-1} \gamma_{n,k,k'}\right| +
 \sup_{m \geq 1} m^{-0.51}
 \left| m - \sum_{k' = -m}^{-1} \gamma_{n,k,k'} \right|.$$
 By classical tail estimates of the Gamma variables, $Z_{n,k} < \infty$ almost surely, and since its law does not depend on $n$ and $k$, $(Z_{n,k})_{n  \geq 1, k \geq 0}$ is a tight family of random variables. Hence, $(Z_n := (Z_{n,k})_{k \geq 0} )_{n \geq 1}$ is a tight family of random variables on $\mathbb{R}^{\mathbb{N}_0}$, endowed with the $\sigma$-algebra generated by the sets
 $\{(z_k)_{k \geq 0}, z_0 \in A_0, z_1 \in A_1, \dots, z_p \in A_p\}$ for $p \geq 0$  and $A_j \in \mathcal{B} (\mathbb{R})$.

 Let us go back to our subsequence of $(\Xi_n^{(k)})_{k \geq 0}$ which converges in law. If we join
 $(\gamma_{n,k, k'})_{k \geq 0, k' \in \mathbb{Z}}$,
$Y_n$ and $Z_n$, we still get a tight family of probability measures on a suitable probability space.  Hence, we can find a sub-subsequence for which the family of random variables
$((\Xi_n^{(k)})_{k \geq 0}, (\gamma_{n,k, k'})_{k \geq 0, k' \in \mathbb{Z}},  Y_n, Z_n)$
converges in law, and a fortiori
$(\Xi_n^{(0)}, (\gamma_{n,k, k'})_{k \geq 0, k' \in \mathbb{Z}},  Y_n, Z_n)$
converges in law.
By Skorokhod representation theorem,
this family has the same law as some family
$(\Xi_n^{'(0)}, (\gamma'_{n,k, k'})_{k \geq 0, k' \in \mathbb{Z}},  Y'_n, Z'_n)$
which converges almost surely along a subsequence.  Note that
$Y'_n$ is function of $\Xi_n^{'(0)}$ and $Z'_n$ is function of the weights $\gamma'_{n,k, k'}$.
Since we know that $\Xi_n^{'(0)}$ converges in law to a $\operatorname{Sine}_{\beta}$ process, its almost sure limit is a simple point measure.

From the boundedness of $Y'_n$ along our subsequence, and Proposition \ref{estimatealmostsureCbeta} (which
implies a bound on the point distribution of the $\operatorname{Sine}_{\beta}$ process, and then the existence of $h_{\ell}$ and its vanishing limit when $\ell \rightarrow \infty$),
we deduce that the part of
Theorem \ref{topologytheorem} concerning
$\Xi_n^{'(0)}$ is satisfied, with
$$h = \underset{\ell \rightarrow \infty}{\lim} \underset{n \rightarrow \infty}{\lim}
h^{sc}_{n, \ell},$$
for
$$h^{sc}_{n, \ell}
= \int_{(-\infty, -\ell] \cup [\ell, \infty)}  \frac{1}{\lambda} d N_{sc} \left(
\alpha \sqrt{n} + \frac{\lambda}{\sqrt{n (4 - \alpha^2)}} \right)$$
We have
\begin{align*}
d N_{sc} \left(
\alpha \sqrt{n} + \frac{\lambda}{\sqrt{n (4 - \alpha^2)}} \right)
& = \frac{n}{2 \pi} d
\left(\int_{-\infty}^{\alpha +\left( \lambda/(n \sqrt{4 - \alpha^2})\right)} \sqrt{(4 - x^2)_+} dx \right)
\\ & = \frac{1}{2 \pi \sqrt{4 - \alpha^2}}
\sqrt{\left[4 - \left(\alpha +\left( \lambda/(n \sqrt{4 - \alpha^2})\right) \right)^2 \right]_+} d \lambda
\end{align*}
If we do a change of variable
$\lambda = \mu n \sqrt{4 - \alpha^2}$, we get
$$h_{n,\ell}^{sc}
= \int_{(-\infty, - \ell/ (n \sqrt{4 - \alpha^2})] \cup [\ell/ (n \sqrt{4 - \alpha^2}), \infty)}
 \frac{1} {2 \pi \sqrt{4 - \alpha^2}}
\sqrt{[4 - (\alpha + \mu)^2]_+} \frac{d \mu}{ \mu}.$$
Taking $n \rightarrow \infty$, we get a quantity independent of $\ell$, given by
$$h = \frac{1}{2 \pi \sqrt{4 - \alpha^2}}
\int_{\mathbb{R}} \sqrt{(4 - y^2)_+}
\frac{dy} {y - \alpha},$$
the integral in the neighborhood of $\alpha$ being understood as a principal value.
From the value of the Stieltjes transform of the semi-circle law, we deduce
$$h = - \frac{\alpha}{2 \sqrt{4  - \alpha^2}}.$$
From the boundedness of $Z'_{n,0}$, we deduce that the part of Theorem \ref{topologytheorem} concerning the weights is also satisfied. Finally, in this theorem, it is almost surely possible to take $\lambda^* = 0$, by the absolutely continuity of the
densities of the ensembles which are considered.

All the assumptions of the theorem are satisfied. If we denote by
$\Lambda_n^{'(0)}$ the measure constructed from $\Xi_n^{'(0)}$ and the weights
 $\gamma'_{n,0,k'}$ ($k' \in \mathbb{Z}$),
 and $\Lambda^{'(0)}$ the measure constructed from the a.s. limits of these points and weights, we deduce that for an independent standard Gaussian variable $g^{(0)}$,
 the function
  $$  z \mapsto - \frac{\alpha}{ \sqrt{4 - \alpha^2}}
- \int \frac{1}{ \lambda- z} d \Lambda_n^{'(0)} (\lambda),$$
and then also the function
 $$ z \mapsto - \frac{\alpha}{ \sqrt{4 - \alpha^2}}
+ \frac{g^{(0)}}{\sqrt{n(4 - \alpha^2)}}
- \frac{z}{n(4-\alpha^2)}
- \int \frac{1}{ \lambda- z} d \Lambda_n^{'(0)} (\lambda),$$
converges uniformly on compact sets, for the topology of the Riemann sphere given in Thorem \ref{topologytheorem}, to
the function
$$ z \mapsto -\frac{\alpha}{ \sqrt{4 - \alpha^2}}
- \int \frac{1}{ \lambda- z} d \Lambda^{'(0)} (\lambda) - h
= -\frac{\alpha}{2 \sqrt{4 - \alpha^2}}
- \int \frac{1}{ \lambda- z} d \Lambda^{'(0)} (\lambda) .$$
As in the proof of Theorem \ref{topologytheorem}, we deduce that the point process
$\Xi_n^{'(1)}$
 given by
$$ - \frac{\alpha}{ \sqrt{4 - \alpha^2}}
+ \frac{g^{(0)}}{\sqrt{n(4 - \alpha^2)}}
- \frac{z}{n(4-\alpha^2)}
- \int \frac{1}{ \lambda- z} d \Lambda_n^{'(0)} (\lambda) = 0,$$
locally weakly converges to
the point process $\Xi^{'(1)}$ given by
$$\underset{\ell \rightarrow \infty}{\lim} \int_{[-\ell, \ell]} \frac{1}{ \lambda- z} d \Lambda^{'(0)} (\lambda) = -\frac{\alpha}{2 \sqrt{4 - \alpha^2}}.$$

The points of $\Xi_n^{'(1)}$ and satisfy the assumptions of Theorem \ref{topologytheorem}, since they interlace with those of $\Xi_n^{'(0)}$. It is also the same for the weights $\gamma'_{n,1,k'}$ ($k' \in \mathbb{Z}$), by the boundedness of $Z'_n$. We then deduce that for an independent Gaussian variable $g^{(1)}$, the point process $\Xi_n^{'(2)}$ given by
$$- \frac{\alpha}{ \sqrt{4 - \alpha^2}}
+ \frac{g^{(1)}}{\sqrt{n(4 - \alpha^2)}}
- \frac{z}{n(4-\alpha^2)}
- \int \frac{1}{ \lambda- z} d \Lambda_n^{'(1)} (\lambda) = 0$$
locally weakly converges
to the process $\Xi^{'(2)}$ given  by
$$\underset{\ell \rightarrow \infty}{\lim} \int_{[-\ell, \ell]} \frac{1}{ \lambda- z} d \Lambda^{'(1)} (\lambda) = -\frac{\alpha}{2 \sqrt{4 - \alpha^2}},$$
where $\Lambda_n^{'(1)}$ is given by $\Xi_n^{'(0)}$ and the weights $\gamma'_{n,1,k'}$ and $\Lambda^{'(1)}$ are given by their limits.
We can then iterate the construction,
which gives a family of point processes
$\Xi_n^{'(k)}$ ($k \geq 0$), converging  to $\Xi^{'(k)}$.
From the way we do this construction, we check that   $(\Xi_n^{'(k)})_{k \geq 0}$ has the same law as  $(\Xi_n^{(k)})_{k \geq 0}$, and that
$\Xi^{'(k)}$ has the same law as the generalized bead process introduced in the present paper (with level lines at $-\alpha/2 \sqrt{4 - \alpha^2}$).
Hence, any subsequence of
$((\Xi_n^{(k)})_{k \geq 0})_{n \geq 1}$
converging in law has a sub-subsequence tending in law to the generalized bead process
By tightness, we deduce the convergence of  the whole sequence $((\Xi_n^{(k)})_{k \geq 0})_{n \geq 1}$.
 This gives the first part of the theorem, after doubling  the weights and the value of $h$.
 The second part is deduced by using the convergence of the GUE minors towards the bead process introduced by
 Boutillier, proven in \cite{ANVM}. The factor $2$ is due to the fact that the average density of points is $1/\pi$
 in \cite{Bou} and $1/2 \pi$ here.
 The value of the parameter $\gamma$ in \cite{Bou} ($a$ in \cite{ANVM}) corresponds to
 $\alpha/2$ (the bulk corresponds to the interval $(-1,1)$ in \cite{ANVM} and to $(-2,2)$ in the present paper).
 We then have
 $$h = - \frac{\alpha}{\sqrt{4 - \alpha^2}} = - \frac{\gamma}{\sqrt{1 - \gamma^2}},$$
 and finally
 $$\gamma = - \frac{h}{\sqrt{1 + h^2}}.$$
\end{proof}
\noindent {\bf Acknowledgments.} B.V. was supported by the Canada
Research Chair program, the NSERC Discovery Accelerator grant, the MTA 
Momentum Random Spectra research group, and the ERC consolidator grant 
648017 (Abert).
\bibliographystyle{halpha}
\bibliography{biblinvmain}

\end{document}